\documentclass{article}%
\usepackage{amsfonts}
\usepackage{amssymb}
\usepackage{amsmath}
\usepackage{amsthm}
\usepackage{fullpage}
\usepackage{xcolor}
\usepackage{url}
\usepackage{graphicx}
\usepackage{graphics}
\usepackage{pstricks}
\usepackage{pst-node}
\usepackage{pst-plot}
\usepackage{float}
\usepackage{varioref}%
\providecommand{\U}[1]{\protect\rule{.1in}{.1in}}
\allowdisplaybreaks
\allowdisplaybreaks
\newtheorem{theorem}{Theorem}
\newtheorem{lemma}[theorem]{Lemma}

\newtheorem{corollary}[theorem]{Corollary}

\newtheorem{example}{Example}
\newtheorem{remark}{Remark}

\begin{document}

\title{Explicit expression for a family of ternary cyclotomic polynomials}
\author{Ala'a Al-Kateeb, Hoon Hong and Eunjeong Lee}
\date{\today}
\maketitle

\begin{abstract}
In this paper, we give an explicit expression for a certain family of ternary
cyclotomic polynomials: specifically $\Phi_{p_{1}p_{2}p_{3}}$, where
$p_{1}<p_{2}<p_{3}$ are odd primes such that $p_{2} \equiv1 \mod p_{1}$ and
$p_{3} \equiv1 \mod {p_{1}p_{2}}$. As an application of the explicit
expression, we give an exact formula for the number of nonzero terms in the
polynomials in the family, which in turn immediately shows that the density
(number of non-zeros terms / degree) is roughly inversely proportional to
$p_{2}$, when $p_{1}$ is sufficiently large.

\end{abstract}

\section{Introduction}

The $n$-th cyclotomic polynomial $\Phi_{n}$ is defined as the monic polynomial
in $\mathbb{Z}[x]$ whose complex roots are the primitive $n$-th~roots of
unity. The cyclotomic polynomials play fundamental roles in number theory and
algebra, with many applications (for instance in
cryptography~\cite{LE2,BW2005,SK,HLL}). Thus they have been extensively
investigated, for instance \cite{BL,BE3,BG1,BZ1,GA-MO,VA,VA2,VA3} on the size of
coefficients, \cite{CA1,BZ4} on the number of non-zero terms,
\cite{HH,Moree2014,Zhang2016,Camburu2016,AH2017} on the maximum size of gaps
in exponents, \cite{AM2,AM1} on efficiently computing coefficients, and so on.

In this paper, we consider the following problem: find an \emph{explicit
expression} for $\Phi_{n}$. By an explicit expression, we mean a polynomial
expression that lists all the terms explicitly. As usual, the problem can be
trivially reduced to the case when $n$ is a product of distinct odd primes.
Thus let us assume, without losing generality, that $n$ is a product of
distinct odd primes, say $n=p_{1}\cdots p_{\ell}$, where $p_{1}<p_{2}%
<\cdots<p_{\ell}$. When $\ell=1$, it is trivial to derive the following
explicit expression
\begin{equation}
\Phi_{p_{1}}=1+x+\cdots+x^{p_{1}-1} \label{l=1}%
\end{equation}
When $\ell=2$, the following expression is hinted or given in
\cite{CA1,BE1,LE1,LL, TH2000,Mor09} using various different notations%
\begin{equation}
\Phi_{p_{1}p_{2}} =\sum_{\substack{0\leq i<s_{1}\\0\leq j<s_{2}}%
}x^{ip_{1}+jp_{2}}\;\;-\;\;\sum_{\substack{0\leq i<p_{2}-s_{1}\\0\leq
j<p_{1}-s_{2}}}x^{ip_{1}+jp_{2}+1} \label{l=2}%
\end{equation}
where $s_{1}\ =p_{1}^{-1}\operatorname{mod}p_{2}\ $ and $s_{2}\ =p_{2}%
^{-1}\operatorname{mod}p_{1}$.
%, equivalently, $p_{1}p_{2}+1=p_{1}\overline{p_{1}}+p_{2}\overline{p_{2}}\ \ $ and $1\leq\overline{p_{1}}\leq p_{2}-1\ $and $1\leq\overline{p_{2}}\leq p_{1}-1$.
The terms in the above expression do not overlap (no cancellation or
accumulation). When $\ell\geq3$, as far as we are aware, there are no general
explicit expressions yet.

The main contribution of this paper is to provide an explicit expression
(Theorem~\ref{thm:exp}) for a certain family of ternary ($\ell=3$) cyclotomic
polynomials: specifically $\Phi_{p_{1}p_{2}p_{3}}$ where%

\begin{equation}
p_{2} \equiv1 \mod p_{1} \ \ \text{and\ \ \ \ } p_{3} \equiv1
\mod {p_{1}p_{2}} \label{family}%
\end{equation}
This family is an interesting one, in that it is flat~\cite{BAC,KA1,KA2},
which means all the coefficients are either $-1$, $0$ or $1$, just like the
above two cases $\ell=1$ and~$\ell=2$.

The provided expression does not have overlapping terms, just like the above
two expressions (\ref{l=1}) and (\ref{l=2}) for the two cases $\ell=1$ and
$\ell=2$. Furthermore, the provided expression is ordered, in that the terms
are ordered in the ascending order in their exponents, just like the
expression (\ref{l=1}) for the case $\ell=1$, but unlike the expression
(\ref{l=2}) for the case $\ell=2$.

Explicit expressions can be useful in extracting various properties of
cyclotomic polynomials, such as coefficients sizes, maximum size of gaps in
exponents (if the expression is also ordered), number of non-zero terms, etc.
In order to illustrate the usefulness, consider the problem of finding an
exact formula for non-zero terms (hamming weight~$\mathrm{hw}$). When $\ell
=1$, it is immediate from the explicit expression~(\ref{l=1}) that%
\[
\mathrm{hw}(\Phi_{p_{1}})=p_{1}%
\]
When $\ell=2$, it is also immediate from the explicit expression~(\ref{l=2})
that
\[
\mathrm{hw}(\Phi_{p_{1}p_{2}})=2\;s_{1}s_{2}-1
\]
(see Carlitz \cite{CA1} for the first but a bit complicated formula). When
$\ell\geq3$, as far as we are aware, there are no general exact formula yet.

In this paper, we derive an exact formula (Corollary \ref{cor:hw}) for
$\mathrm{hw}(\Phi_{p_{1}p_{2}p_{3}})$ for the family of the cyclotomic
polynomials satisfying (\ref{family}), by crucially exploiting the explicit
expression. The exact formula in turn immediately shows that the density
(number of non-zeros terms / degree) is roughly inversely proportional to
$p_{2}$, when $p_{1}$ is sufficiently large (Corollary~\ref{cor:density}).

The paper is structured as follows. In Section \ref{sec:results}, we precisely
state the main result and its corollaries. In Section \ref{sec:proof}, we
prove the main result.

\section{Main Results}

\label{sec:results} In this section, we will give a precise statement of the
main result (Theorem~\ref{thm:exp}) and illustrate its usefulness by a couple
of applications (Corollaries~\ref{cor:hw} and \ref{cor:density}).

\begin{theorem}
[Explicit expression]\label{thm:exp} Let $p_{2}\equiv1\operatorname{mod}p_{1}$
and $p_{3}\equiv1\operatorname{mod}p_{1}p_{2}$. We make two claims.

\begin{enumerate}
\item[\textsf{C1.}] The polynomial $\Phi_{p_{1}p_{2}p_{3}}$ can be explicitly
written as follows.
\begin{align*}
\Phi_{p_{1}p_{2}p_{3}}  &  =\sum_{i\in I}f_{i}\ x^{i\cdot\rho}%
\ \ +\ \ x^{\varphi\left(  p_{1}p_{2}p_{3}\right)  }\\
f_{i}  &  =1+g_{i}-x^{\left(  i_{1}+1\right)  p_{2}}-x^{p_{2}}\ g_{i}\\
g_{i}  &  =\left\{
\begin{array}
[c]{llll}%
\displaystyle+x^{i_{1}p_{2}+i_{2}p_{1}+1}\;\;\;\;\sum_{k=0}^{p_{1}-2-i_{1}} &
x^{k} & \text{if } & i_{3}\leq i_{1}\\
\displaystyle-x^{i_{1}\left(  p_{2}-1\right)  +\left(  i_{2}+1\right)  p_{1}%
}\sum_{k=0}^{i_{1}} & x^{k} & \text{if} & i_{3}>i_{1}%
\end{array}
\right.
\end{align*}
where%
\begin{align*}
I  &  =\left\{  0,\ldots,p_{1}-2\right\}  \times\left\{  0,\ldots
,q_{2}-1\right\}  \times\left\{  0,\ldots,p_{1}-1\right\}  \times\left\{
0,\ldots,q_{3}-1\right\} \\
\rho &  =\left(  \left(  p_{2}-1\right)  \left(  p_{3}-1\right)
,\ p_{1}\left(  p_{3}-1\right)  ,\ p_{3}-1,\ p_{1}p_{2}\right) \\
q_{2}  &  ={\mathrm{quo}}\left(  p_{2},p_{1}\right)  \ \ \ q_{3}%
={\mathrm{quo}}\left(  p_{3},p_{1}p_{2}\right)
\end{align*}

\item[\textsf{C2.}] The expression does not have any overlapping (cancellation
or accumulation)\ of terms and their exponents are ordered in the ascending
order when $\sum_{i\in I}$ is carried out as $\sum_{i_{1}}\sum_{i_{2}}%
\sum_{i_{3}}\sum_{i_{4}}$.
\end{enumerate}
\end{theorem}

\begin{example}
\label{exm:exp}We will illustrate Theorem \ref{thm:exp} by using a small
example: $p_{1}=3,\;p_{2}=13$ and $p_{3}=79$. It is obvious that $p_{2}%
\equiv1\operatorname{mod}p_{1}$ and $p_{3}\equiv1\operatorname{mod}p_{1}p_{2}$.

\begin{enumerate}
\item[\textsf{C1.}] An explicit expression for $\Phi_{3\cdot13\cdot79}$ can be
obtained from Theorem \ref{thm:exp}-\textsf{C1} as follows. Note that
$q_{2}=4$ and $q_{3}=2$. Hence
\[
\Phi_{3\cdot13\cdot79}=\sum_{i_{1}=0}^{1}\sum_{i_{2}=0}^{3}\sum_{i_{3}=0}%
^{2}\sum_{i_{4}=0}^{1}\ f_{i}\ \ x^{i_{1}\rho_{1}+i_{2}\rho_{2}+i_{3}\rho
_{3}+i_{4}\rho_{4}}\ \ +\ \ x^{\varphi\left(  3\cdot13\cdot79\right)  }%
\]
where%
\[
\rho=\left(  \left(  13-1\right)  \left(  79-1\right)  ,\ 3\cdot\left(
79-1\right)  ,\ 79-1,\ 3\cdot13\right)  =\left(  936,234,78,39\right)
\]
In other words, the polynomial $\Phi_{3\cdot13\cdot79}$ can be represented
under the multi-radices $\rho$ as
\[
\Phi_{3\cdot13\cdot79}=%
\begin{array}
[c]{|c|c|c|c|c|c|c|c|c|c|c|c|}\hline
f_{0000} & f_{0001} & f_{0010} & f_{0011} & f_{0020} & f_{0021} & f_{0100} &
f_{0101} & \cdots & f_{1320} & f_{1321} & 1\\\hline
\end{array}
\]
where the \textquotedblleft digits\textquotedblright\ $f_{i}$'s are given by
\definecolor{Color1}{rgb}{1.0, 0.0,0.0}
\definecolor{Color2}{rgb}{0.0, 1.0,0.0}
\definecolor{Color3}{rgb}{1.0, 0.0,1.0}
\definecolor{Color4}{rgb}{0.0, 1.0,1.0}
{\tiny\[
\begin{array}{|rrrr|l|} \hline
i_1&i_2&i_3&i_4&                                                                                                 f_i\\ \hline
\hline  0&  0&  0&  0&
{\psset{unit=0.15} \pspicture[shift=*](0,0)(39,0.45)
\psframe[linewidth=0.1pt](0.0,0)(39,1)
\psframe[linewidth=0.1pt,fillstyle=solid,fillcolor=Color1](0,0)(1,1)
\psline[linewidth=0.4pt,linecolor=white](.2,0.5)(.8,0.5)
\psline[linewidth=0.4pt,linecolor=white](.5,0.2)(.5,0.8)
\psframe[linewidth=0.1pt,fillstyle=solid,fillcolor=Color2](1,0)(3,1)
\psline[linewidth=0.4pt,linecolor=white](1.2,0.5)(1.8,0.5)
\psline[linewidth=0.4pt,linecolor=white](1.5,0.2)(1.5,0.8)
\psline[linewidth=0.4pt,linecolor=white](2.2,0.5)(2.8,0.5)
\psline[linewidth=0.4pt,linecolor=white](2.5,0.2)(2.5,0.8)
\psframe[linewidth=0.1pt,fillstyle=solid,fillcolor=Color3](13,0)(14,1)
\psline[linewidth=0.4pt,linecolor=white](13.2,0.5)(13.8,0.5)
\psframe[linewidth=0.1pt,fillstyle=solid,fillcolor=Color4](14,0)(16,1)
\psline[linewidth=0.4pt,linecolor=white](14.2,0.5)(14.8,0.5)
\psline[linewidth=0.4pt,linecolor=white](15.2,0.5)(15.8,0.5)
\multips(0,0)(1,0){39}{\psline[linewidth=0.1pt](0,0)(0,1)}
\endpspicture }\\
  0&  0&  0&  1&
{\psset{unit=0.15} \pspicture[shift=*](0,0)(39,0.45)
\psframe[linewidth=0.1pt](0.0,0)(39,1)
\psframe[linewidth=0.1pt,fillstyle=solid,fillcolor=Color1](0,0)(1,1)
\psline[linewidth=0.4pt,linecolor=white](.2,0.5)(.8,0.5)
\psline[linewidth=0.4pt,linecolor=white](.5,0.2)(.5,0.8)
\psframe[linewidth=0.1pt,fillstyle=solid,fillcolor=Color2](1,0)(3,1)
\psline[linewidth=0.4pt,linecolor=white](1.2,0.5)(1.8,0.5)
\psline[linewidth=0.4pt,linecolor=white](1.5,0.2)(1.5,0.8)
\psline[linewidth=0.4pt,linecolor=white](2.2,0.5)(2.8,0.5)
\psline[linewidth=0.4pt,linecolor=white](2.5,0.2)(2.5,0.8)
\psframe[linewidth=0.1pt,fillstyle=solid,fillcolor=Color3](13,0)(14,1)
\psline[linewidth=0.4pt,linecolor=white](13.2,0.5)(13.8,0.5)
\psframe[linewidth=0.1pt,fillstyle=solid,fillcolor=Color4](14,0)(16,1)
\psline[linewidth=0.4pt,linecolor=white](14.2,0.5)(14.8,0.5)
\psline[linewidth=0.4pt,linecolor=white](15.2,0.5)(15.8,0.5)
\multips(0,0)(1,0){39}{\psline[linewidth=0.1pt](0,0)(0,1)}
\endpspicture }\\
\hline  0&  0&  1&  0&
{\psset{unit=0.15} \pspicture[shift=*](0,0)(39,0.45)
\psframe[linewidth=0.1pt](0.0,0)(39,1)
\psframe[linewidth=0.1pt,fillstyle=solid,fillcolor=Color1](0,0)(1,1)
\psline[linewidth=0.4pt,linecolor=white](.2,0.5)(.8,0.5)
\psline[linewidth=0.4pt,linecolor=white](.5,0.2)(.5,0.8)
\psframe[linewidth=0.1pt,fillstyle=solid,fillcolor=Color2](3,0)(4,1)
\psline[linewidth=0.4pt,linecolor=white](3.2,0.5)(3.8,0.5)
\psframe[linewidth=0.1pt,fillstyle=solid,fillcolor=Color3](13,0)(14,1)
\psline[linewidth=0.4pt,linecolor=white](13.2,0.5)(13.8,0.5)
\psframe[linewidth=0.1pt,fillstyle=solid,fillcolor=Color4](16,0)(17,1)
\psline[linewidth=0.4pt,linecolor=white](16.2,0.5)(16.8,0.5)
\psline[linewidth=0.4pt,linecolor=white](16.5,0.2)(16.5,0.8)
\multips(0,0)(1,0){39}{\psline[linewidth=0.1pt](0,0)(0,1)}
\endpspicture }\\
  0&  0&  1&  1&
{\psset{unit=0.15} \pspicture[shift=*](0,0)(39,0.45)
\psframe[linewidth=0.1pt](0.0,0)(39,1)
\psframe[linewidth=0.1pt,fillstyle=solid,fillcolor=Color1](0,0)(1,1)
\psline[linewidth=0.4pt,linecolor=white](.2,0.5)(.8,0.5)
\psline[linewidth=0.4pt,linecolor=white](.5,0.2)(.5,0.8)
\psframe[linewidth=0.1pt,fillstyle=solid,fillcolor=Color2](3,0)(4,1)
\psline[linewidth=0.4pt,linecolor=white](3.2,0.5)(3.8,0.5)
\psframe[linewidth=0.1pt,fillstyle=solid,fillcolor=Color3](13,0)(14,1)
\psline[linewidth=0.4pt,linecolor=white](13.2,0.5)(13.8,0.5)
\psframe[linewidth=0.1pt,fillstyle=solid,fillcolor=Color4](16,0)(17,1)
\psline[linewidth=0.4pt,linecolor=white](16.2,0.5)(16.8,0.5)
\psline[linewidth=0.4pt,linecolor=white](16.5,0.2)(16.5,0.8)
\multips(0,0)(1,0){39}{\psline[linewidth=0.1pt](0,0)(0,1)}
\endpspicture }\\
  0&  0&  2&  0&
{\psset{unit=0.15} \pspicture[shift=*](0,0)(39,0.45)
\psframe[linewidth=0.1pt](0.0,0)(39,1)
\psframe[linewidth=0.1pt,fillstyle=solid,fillcolor=Color1](0,0)(1,1)
\psline[linewidth=0.4pt,linecolor=white](.2,0.5)(.8,0.5)
\psline[linewidth=0.4pt,linecolor=white](.5,0.2)(.5,0.8)
\psframe[linewidth=0.1pt,fillstyle=solid,fillcolor=Color2](3,0)(4,1)
\psline[linewidth=0.4pt,linecolor=white](3.2,0.5)(3.8,0.5)
\psframe[linewidth=0.1pt,fillstyle=solid,fillcolor=Color3](13,0)(14,1)
\psline[linewidth=0.4pt,linecolor=white](13.2,0.5)(13.8,0.5)
\psframe[linewidth=0.1pt,fillstyle=solid,fillcolor=Color4](16,0)(17,1)
\psline[linewidth=0.4pt,linecolor=white](16.2,0.5)(16.8,0.5)
\psline[linewidth=0.4pt,linecolor=white](16.5,0.2)(16.5,0.8)
\multips(0,0)(1,0){39}{\psline[linewidth=0.1pt](0,0)(0,1)}
\endpspicture }\\
  0&  0&  2&  1&
{\psset{unit=0.15} \pspicture[shift=*](0,0)(39,0.45)
\psframe[linewidth=0.1pt](0.0,0)(39,1)
\psframe[linewidth=0.1pt,fillstyle=solid,fillcolor=Color1](0,0)(1,1)
\psline[linewidth=0.4pt,linecolor=white](.2,0.5)(.8,0.5)
\psline[linewidth=0.4pt,linecolor=white](.5,0.2)(.5,0.8)
\psframe[linewidth=0.1pt,fillstyle=solid,fillcolor=Color2](3,0)(4,1)
\psline[linewidth=0.4pt,linecolor=white](3.2,0.5)(3.8,0.5)
\psframe[linewidth=0.1pt,fillstyle=solid,fillcolor=Color3](13,0)(14,1)
\psline[linewidth=0.4pt,linecolor=white](13.2,0.5)(13.8,0.5)
\psframe[linewidth=0.1pt,fillstyle=solid,fillcolor=Color4](16,0)(17,1)
\psline[linewidth=0.4pt,linecolor=white](16.2,0.5)(16.8,0.5)
\psline[linewidth=0.4pt,linecolor=white](16.5,0.2)(16.5,0.8)
\multips(0,0)(1,0){39}{\psline[linewidth=0.1pt](0,0)(0,1)}
\endpspicture }\\
\hline  0&  1&  0&  0&
{\psset{unit=0.15} \pspicture[shift=*](0,0)(39,0.45)
\psframe[linewidth=0.1pt](0.0,0)(39,1)
\psframe[linewidth=0.1pt,fillstyle=solid,fillcolor=Color1](0,0)(1,1)
\psline[linewidth=0.4pt,linecolor=white](.2,0.5)(.8,0.5)
\psline[linewidth=0.4pt,linecolor=white](.5,0.2)(.5,0.8)
\psframe[linewidth=0.1pt,fillstyle=solid,fillcolor=Color2](4,0)(6,1)
\psline[linewidth=0.4pt,linecolor=white](4.2,0.5)(4.8,0.5)
\psline[linewidth=0.4pt,linecolor=white](4.5,0.2)(4.5,0.8)
\psline[linewidth=0.4pt,linecolor=white](5.2,0.5)(5.8,0.5)
\psline[linewidth=0.4pt,linecolor=white](5.5,0.2)(5.5,0.8)
\psframe[linewidth=0.1pt,fillstyle=solid,fillcolor=Color3](13,0)(14,1)
\psline[linewidth=0.4pt,linecolor=white](13.2,0.5)(13.8,0.5)
\psframe[linewidth=0.1pt,fillstyle=solid,fillcolor=Color4](17,0)(19,1)
\psline[linewidth=0.4pt,linecolor=white](17.2,0.5)(17.8,0.5)
\psline[linewidth=0.4pt,linecolor=white](18.2,0.5)(18.8,0.5)
\multips(0,0)(1,0){39}{\psline[linewidth=0.1pt](0,0)(0,1)}
\endpspicture }\\
  0&  1&  0&  1&
{\psset{unit=0.15} \pspicture[shift=*](0,0)(39,0.45)
\psframe[linewidth=0.1pt](0.0,0)(39,1)
\psframe[linewidth=0.1pt,fillstyle=solid,fillcolor=Color1](0,0)(1,1)
\psline[linewidth=0.4pt,linecolor=white](.2,0.5)(.8,0.5)
\psline[linewidth=0.4pt,linecolor=white](.5,0.2)(.5,0.8)
\psframe[linewidth=0.1pt,fillstyle=solid,fillcolor=Color2](4,0)(6,1)
\psline[linewidth=0.4pt,linecolor=white](4.2,0.5)(4.8,0.5)
\psline[linewidth=0.4pt,linecolor=white](4.5,0.2)(4.5,0.8)
\psline[linewidth=0.4pt,linecolor=white](5.2,0.5)(5.8,0.5)
\psline[linewidth=0.4pt,linecolor=white](5.5,0.2)(5.5,0.8)
\psframe[linewidth=0.1pt,fillstyle=solid,fillcolor=Color3](13,0)(14,1)
\psline[linewidth=0.4pt,linecolor=white](13.2,0.5)(13.8,0.5)
\psframe[linewidth=0.1pt,fillstyle=solid,fillcolor=Color4](17,0)(19,1)
\psline[linewidth=0.4pt,linecolor=white](17.2,0.5)(17.8,0.5)
\psline[linewidth=0.4pt,linecolor=white](18.2,0.5)(18.8,0.5)
\multips(0,0)(1,0){39}{\psline[linewidth=0.1pt](0,0)(0,1)}
\endpspicture }\\
\hline  0&  1&  1&  0&
{\psset{unit=0.15} \pspicture[shift=*](0,0)(39,0.45)
\psframe[linewidth=0.1pt](0.0,0)(39,1)
\psframe[linewidth=0.1pt,fillstyle=solid,fillcolor=Color1](0,0)(1,1)
\psline[linewidth=0.4pt,linecolor=white](.2,0.5)(.8,0.5)
\psline[linewidth=0.4pt,linecolor=white](.5,0.2)(.5,0.8)
\psframe[linewidth=0.1pt,fillstyle=solid,fillcolor=Color2](6,0)(7,1)
\psline[linewidth=0.4pt,linecolor=white](6.2,0.5)(6.8,0.5)
\psframe[linewidth=0.1pt,fillstyle=solid,fillcolor=Color3](13,0)(14,1)
\psline[linewidth=0.4pt,linecolor=white](13.2,0.5)(13.8,0.5)
\psframe[linewidth=0.1pt,fillstyle=solid,fillcolor=Color4](19,0)(20,1)
\psline[linewidth=0.4pt,linecolor=white](19.2,0.5)(19.8,0.5)
\psline[linewidth=0.4pt,linecolor=white](19.5,0.2)(19.5,0.8)
\multips(0,0)(1,0){39}{\psline[linewidth=0.1pt](0,0)(0,1)}
\endpspicture }\\
  0&  1&  1&  1&
{\psset{unit=0.15} \pspicture[shift=*](0,0)(39,0.45)
\psframe[linewidth=0.1pt](0.0,0)(39,1)
\psframe[linewidth=0.1pt,fillstyle=solid,fillcolor=Color1](0,0)(1,1)
\psline[linewidth=0.4pt,linecolor=white](.2,0.5)(.8,0.5)
\psline[linewidth=0.4pt,linecolor=white](.5,0.2)(.5,0.8)
\psframe[linewidth=0.1pt,fillstyle=solid,fillcolor=Color2](6,0)(7,1)
\psline[linewidth=0.4pt,linecolor=white](6.2,0.5)(6.8,0.5)
\psframe[linewidth=0.1pt,fillstyle=solid,fillcolor=Color3](13,0)(14,1)
\psline[linewidth=0.4pt,linecolor=white](13.2,0.5)(13.8,0.5)
\psframe[linewidth=0.1pt,fillstyle=solid,fillcolor=Color4](19,0)(20,1)
\psline[linewidth=0.4pt,linecolor=white](19.2,0.5)(19.8,0.5)
\psline[linewidth=0.4pt,linecolor=white](19.5,0.2)(19.5,0.8)
\multips(0,0)(1,0){39}{\psline[linewidth=0.1pt](0,0)(0,1)}
\endpspicture }\\
  0&  1&  2&  0&
{\psset{unit=0.15} \pspicture[shift=*](0,0)(39,0.45)
\psframe[linewidth=0.1pt](0.0,0)(39,1)
\psframe[linewidth=0.1pt,fillstyle=solid,fillcolor=Color1](0,0)(1,1)
\psline[linewidth=0.4pt,linecolor=white](.2,0.5)(.8,0.5)
\psline[linewidth=0.4pt,linecolor=white](.5,0.2)(.5,0.8)
\psframe[linewidth=0.1pt,fillstyle=solid,fillcolor=Color2](6,0)(7,1)
\psline[linewidth=0.4pt,linecolor=white](6.2,0.5)(6.8,0.5)
\psframe[linewidth=0.1pt,fillstyle=solid,fillcolor=Color3](13,0)(14,1)
\psline[linewidth=0.4pt,linecolor=white](13.2,0.5)(13.8,0.5)
\psframe[linewidth=0.1pt,fillstyle=solid,fillcolor=Color4](19,0)(20,1)
\psline[linewidth=0.4pt,linecolor=white](19.2,0.5)(19.8,0.5)
\psline[linewidth=0.4pt,linecolor=white](19.5,0.2)(19.5,0.8)
\multips(0,0)(1,0){39}{\psline[linewidth=0.1pt](0,0)(0,1)}
\endpspicture }\\
  0&  1&  2&  1&
{\psset{unit=0.15} \pspicture[shift=*](0,0)(39,0.45)
\psframe[linewidth=0.1pt](0.0,0)(39,1)
\psframe[linewidth=0.1pt,fillstyle=solid,fillcolor=Color1](0,0)(1,1)
\psline[linewidth=0.4pt,linecolor=white](.2,0.5)(.8,0.5)
\psline[linewidth=0.4pt,linecolor=white](.5,0.2)(.5,0.8)
\psframe[linewidth=0.1pt,fillstyle=solid,fillcolor=Color2](6,0)(7,1)
\psline[linewidth=0.4pt,linecolor=white](6.2,0.5)(6.8,0.5)
\psframe[linewidth=0.1pt,fillstyle=solid,fillcolor=Color3](13,0)(14,1)
\psline[linewidth=0.4pt,linecolor=white](13.2,0.5)(13.8,0.5)
\psframe[linewidth=0.1pt,fillstyle=solid,fillcolor=Color4](19,0)(20,1)
\psline[linewidth=0.4pt,linecolor=white](19.2,0.5)(19.8,0.5)
\psline[linewidth=0.4pt,linecolor=white](19.5,0.2)(19.5,0.8)
\multips(0,0)(1,0){39}{\psline[linewidth=0.1pt](0,0)(0,1)}
\endpspicture }\\
\hline  0&  2&  0&  0&
{\psset{unit=0.15} \pspicture[shift=*](0,0)(39,0.45)
\psframe[linewidth=0.1pt](0.0,0)(39,1)
\psframe[linewidth=0.1pt,fillstyle=solid,fillcolor=Color1](0,0)(1,1)
\psline[linewidth=0.4pt,linecolor=white](.2,0.5)(.8,0.5)
\psline[linewidth=0.4pt,linecolor=white](.5,0.2)(.5,0.8)
\psframe[linewidth=0.1pt,fillstyle=solid,fillcolor=Color2](7,0)(9,1)
\psline[linewidth=0.4pt,linecolor=white](7.2,0.5)(7.8,0.5)
\psline[linewidth=0.4pt,linecolor=white](7.5,0.2)(7.5,0.8)
\psline[linewidth=0.4pt,linecolor=white](8.2,0.5)(8.8,0.5)
\psline[linewidth=0.4pt,linecolor=white](8.5,0.2)(8.5,0.8)
\psframe[linewidth=0.1pt,fillstyle=solid,fillcolor=Color3](13,0)(14,1)
\psline[linewidth=0.4pt,linecolor=white](13.2,0.5)(13.8,0.5)
\psframe[linewidth=0.1pt,fillstyle=solid,fillcolor=Color4](20,0)(22,1)
\psline[linewidth=0.4pt,linecolor=white](20.2,0.5)(20.8,0.5)
\psline[linewidth=0.4pt,linecolor=white](21.2,0.5)(21.8,0.5)
\multips(0,0)(1,0){39}{\psline[linewidth=0.1pt](0,0)(0,1)}
\endpspicture }\\
  0&  2&  0&  1&
{\psset{unit=0.15} \pspicture[shift=*](0,0)(39,0.45)
\psframe[linewidth=0.1pt](0.0,0)(39,1)
\psframe[linewidth=0.1pt,fillstyle=solid,fillcolor=Color1](0,0)(1,1)
\psline[linewidth=0.4pt,linecolor=white](.2,0.5)(.8,0.5)
\psline[linewidth=0.4pt,linecolor=white](.5,0.2)(.5,0.8)
\psframe[linewidth=0.1pt,fillstyle=solid,fillcolor=Color2](7,0)(9,1)
\psline[linewidth=0.4pt,linecolor=white](7.2,0.5)(7.8,0.5)
\psline[linewidth=0.4pt,linecolor=white](7.5,0.2)(7.5,0.8)
\psline[linewidth=0.4pt,linecolor=white](8.2,0.5)(8.8,0.5)
\psline[linewidth=0.4pt,linecolor=white](8.5,0.2)(8.5,0.8)
\psframe[linewidth=0.1pt,fillstyle=solid,fillcolor=Color3](13,0)(14,1)
\psline[linewidth=0.4pt,linecolor=white](13.2,0.5)(13.8,0.5)
\psframe[linewidth=0.1pt,fillstyle=solid,fillcolor=Color4](20,0)(22,1)
\psline[linewidth=0.4pt,linecolor=white](20.2,0.5)(20.8,0.5)
\psline[linewidth=0.4pt,linecolor=white](21.2,0.5)(21.8,0.5)
\multips(0,0)(1,0){39}{\psline[linewidth=0.1pt](0,0)(0,1)}
\endpspicture }\\
\hline  0&  2&  1&  0&
{\psset{unit=0.15} \pspicture[shift=*](0,0)(39,0.45)
\psframe[linewidth=0.1pt](0.0,0)(39,1)
\psframe[linewidth=0.1pt,fillstyle=solid,fillcolor=Color1](0,0)(1,1)
\psline[linewidth=0.4pt,linecolor=white](.2,0.5)(.8,0.5)
\psline[linewidth=0.4pt,linecolor=white](.5,0.2)(.5,0.8)
\psframe[linewidth=0.1pt,fillstyle=solid,fillcolor=Color2](9,0)(10,1)
\psline[linewidth=0.4pt,linecolor=white](9.2,0.5)(9.8,0.5)
\psframe[linewidth=0.1pt,fillstyle=solid,fillcolor=Color3](13,0)(14,1)
\psline[linewidth=0.4pt,linecolor=white](13.2,0.5)(13.8,0.5)
\psframe[linewidth=0.1pt,fillstyle=solid,fillcolor=Color4](22,0)(23,1)
\psline[linewidth=0.4pt,linecolor=white](22.2,0.5)(22.8,0.5)
\psline[linewidth=0.4pt,linecolor=white](22.5,0.2)(22.5,0.8)
\multips(0,0)(1,0){39}{\psline[linewidth=0.1pt](0,0)(0,1)}
\endpspicture }\\
  0&  2&  1&  1&
{\psset{unit=0.15} \pspicture[shift=*](0,0)(39,0.45)
\psframe[linewidth=0.1pt](0.0,0)(39,1)
\psframe[linewidth=0.1pt,fillstyle=solid,fillcolor=Color1](0,0)(1,1)
\psline[linewidth=0.4pt,linecolor=white](.2,0.5)(.8,0.5)
\psline[linewidth=0.4pt,linecolor=white](.5,0.2)(.5,0.8)
\psframe[linewidth=0.1pt,fillstyle=solid,fillcolor=Color2](9,0)(10,1)
\psline[linewidth=0.4pt,linecolor=white](9.2,0.5)(9.8,0.5)
\psframe[linewidth=0.1pt,fillstyle=solid,fillcolor=Color3](13,0)(14,1)
\psline[linewidth=0.4pt,linecolor=white](13.2,0.5)(13.8,0.5)
\psframe[linewidth=0.1pt,fillstyle=solid,fillcolor=Color4](22,0)(23,1)
\psline[linewidth=0.4pt,linecolor=white](22.2,0.5)(22.8,0.5)
\psline[linewidth=0.4pt,linecolor=white](22.5,0.2)(22.5,0.8)
\multips(0,0)(1,0){39}{\psline[linewidth=0.1pt](0,0)(0,1)}
\endpspicture }\\
  0&  2&  2&  0&
{\psset{unit=0.15} \pspicture[shift=*](0,0)(39,0.45)
\psframe[linewidth=0.1pt](0.0,0)(39,1)
\psframe[linewidth=0.1pt,fillstyle=solid,fillcolor=Color1](0,0)(1,1)
\psline[linewidth=0.4pt,linecolor=white](.2,0.5)(.8,0.5)
\psline[linewidth=0.4pt,linecolor=white](.5,0.2)(.5,0.8)
\psframe[linewidth=0.1pt,fillstyle=solid,fillcolor=Color2](9,0)(10,1)
\psline[linewidth=0.4pt,linecolor=white](9.2,0.5)(9.8,0.5)
\psframe[linewidth=0.1pt,fillstyle=solid,fillcolor=Color3](13,0)(14,1)
\psline[linewidth=0.4pt,linecolor=white](13.2,0.5)(13.8,0.5)
\psframe[linewidth=0.1pt,fillstyle=solid,fillcolor=Color4](22,0)(23,1)
\psline[linewidth=0.4pt,linecolor=white](22.2,0.5)(22.8,0.5)
\psline[linewidth=0.4pt,linecolor=white](22.5,0.2)(22.5,0.8)
\multips(0,0)(1,0){39}{\psline[linewidth=0.1pt](0,0)(0,1)}
\endpspicture }\\
  0&  2&  2&  1&
{\psset{unit=0.15} \pspicture[shift=*](0,0)(39,0.45)
\psframe[linewidth=0.1pt](0.0,0)(39,1)
\psframe[linewidth=0.1pt,fillstyle=solid,fillcolor=Color1](0,0)(1,1)
\psline[linewidth=0.4pt,linecolor=white](.2,0.5)(.8,0.5)
\psline[linewidth=0.4pt,linecolor=white](.5,0.2)(.5,0.8)
\psframe[linewidth=0.1pt,fillstyle=solid,fillcolor=Color2](9,0)(10,1)
\psline[linewidth=0.4pt,linecolor=white](9.2,0.5)(9.8,0.5)
\psframe[linewidth=0.1pt,fillstyle=solid,fillcolor=Color3](13,0)(14,1)
\psline[linewidth=0.4pt,linecolor=white](13.2,0.5)(13.8,0.5)
\psframe[linewidth=0.1pt,fillstyle=solid,fillcolor=Color4](22,0)(23,1)
\psline[linewidth=0.4pt,linecolor=white](22.2,0.5)(22.8,0.5)
\psline[linewidth=0.4pt,linecolor=white](22.5,0.2)(22.5,0.8)
\multips(0,0)(1,0){39}{\psline[linewidth=0.1pt](0,0)(0,1)}
\endpspicture }\\
\hline  0&  3&  0&  0&
{\psset{unit=0.15} \pspicture[shift=*](0,0)(39,0.45)
\psframe[linewidth=0.1pt](0.0,0)(39,1)
\psframe[linewidth=0.1pt,fillstyle=solid,fillcolor=Color1](0,0)(1,1)
\psline[linewidth=0.4pt,linecolor=white](.2,0.5)(.8,0.5)
\psline[linewidth=0.4pt,linecolor=white](.5,0.2)(.5,0.8)
\psframe[linewidth=0.1pt,fillstyle=solid,fillcolor=Color2](10,0)(12,1)
\psline[linewidth=0.4pt,linecolor=white](10.2,0.5)(10.8,0.5)
\psline[linewidth=0.4pt,linecolor=white](10.5,0.2)(10.5,0.8)
\psline[linewidth=0.4pt,linecolor=white](11.2,0.5)(11.8,0.5)
\psline[linewidth=0.4pt,linecolor=white](11.5,0.2)(11.5,0.8)
\psframe[linewidth=0.1pt,fillstyle=solid,fillcolor=Color3](13,0)(14,1)
\psline[linewidth=0.4pt,linecolor=white](13.2,0.5)(13.8,0.5)
\psframe[linewidth=0.1pt,fillstyle=solid,fillcolor=Color4](23,0)(25,1)
\psline[linewidth=0.4pt,linecolor=white](23.2,0.5)(23.8,0.5)
\psline[linewidth=0.4pt,linecolor=white](24.2,0.5)(24.8,0.5)
\multips(0,0)(1,0){39}{\psline[linewidth=0.1pt](0,0)(0,1)}
\endpspicture }\\
  0&  3&  0&  1&
{\psset{unit=0.15} \pspicture[shift=*](0,0)(39,0.45)
\psframe[linewidth=0.1pt](0.0,0)(39,1)
\psframe[linewidth=0.1pt,fillstyle=solid,fillcolor=Color1](0,0)(1,1)
\psline[linewidth=0.4pt,linecolor=white](.2,0.5)(.8,0.5)
\psline[linewidth=0.4pt,linecolor=white](.5,0.2)(.5,0.8)
\psframe[linewidth=0.1pt,fillstyle=solid,fillcolor=Color2](10,0)(12,1)
\psline[linewidth=0.4pt,linecolor=white](10.2,0.5)(10.8,0.5)
\psline[linewidth=0.4pt,linecolor=white](10.5,0.2)(10.5,0.8)
\psline[linewidth=0.4pt,linecolor=white](11.2,0.5)(11.8,0.5)
\psline[linewidth=0.4pt,linecolor=white](11.5,0.2)(11.5,0.8)
\psframe[linewidth=0.1pt,fillstyle=solid,fillcolor=Color3](13,0)(14,1)
\psline[linewidth=0.4pt,linecolor=white](13.2,0.5)(13.8,0.5)
\psframe[linewidth=0.1pt,fillstyle=solid,fillcolor=Color4](23,0)(25,1)
\psline[linewidth=0.4pt,linecolor=white](23.2,0.5)(23.8,0.5)
\psline[linewidth=0.4pt,linecolor=white](24.2,0.5)(24.8,0.5)
\multips(0,0)(1,0){39}{\psline[linewidth=0.1pt](0,0)(0,1)}
\endpspicture }\\
\hline  0&  3&  1&  0&
{\psset{unit=0.15} \pspicture[shift=*](0,0)(39,0.45)
\psframe[linewidth=0.1pt](0.0,0)(39,1)
\psframe[linewidth=0.1pt,fillstyle=solid,fillcolor=Color1](0,0)(1,1)
\psline[linewidth=0.4pt,linecolor=white](.2,0.5)(.8,0.5)
\psline[linewidth=0.4pt,linecolor=white](.5,0.2)(.5,0.8)
\psframe[linewidth=0.1pt,fillstyle=solid,fillcolor=Color2](12,0)(13,1)
\psline[linewidth=0.4pt,linecolor=white](12.2,0.5)(12.8,0.5)
\psframe[linewidth=0.1pt,fillstyle=solid,fillcolor=Color3](13,0)(14,1)
\psline[linewidth=0.4pt,linecolor=white](13.2,0.5)(13.8,0.5)
\psframe[linewidth=0.1pt,fillstyle=solid,fillcolor=Color4](25,0)(26,1)
\psline[linewidth=0.4pt,linecolor=white](25.2,0.5)(25.8,0.5)
\psline[linewidth=0.4pt,linecolor=white](25.5,0.2)(25.5,0.8)
\multips(0,0)(1,0){39}{\psline[linewidth=0.1pt](0,0)(0,1)}
\endpspicture }\\
  0&  3&  1&  1&
{\psset{unit=0.15} \pspicture[shift=*](0,0)(39,0.45)
\psframe[linewidth=0.1pt](0.0,0)(39,1)
\psframe[linewidth=0.1pt,fillstyle=solid,fillcolor=Color1](0,0)(1,1)
\psline[linewidth=0.4pt,linecolor=white](.2,0.5)(.8,0.5)
\psline[linewidth=0.4pt,linecolor=white](.5,0.2)(.5,0.8)
\psframe[linewidth=0.1pt,fillstyle=solid,fillcolor=Color2](12,0)(13,1)
\psline[linewidth=0.4pt,linecolor=white](12.2,0.5)(12.8,0.5)
\psframe[linewidth=0.1pt,fillstyle=solid,fillcolor=Color3](13,0)(14,1)
\psline[linewidth=0.4pt,linecolor=white](13.2,0.5)(13.8,0.5)
\psframe[linewidth=0.1pt,fillstyle=solid,fillcolor=Color4](25,0)(26,1)
\psline[linewidth=0.4pt,linecolor=white](25.2,0.5)(25.8,0.5)
\psline[linewidth=0.4pt,linecolor=white](25.5,0.2)(25.5,0.8)
\multips(0,0)(1,0){39}{\psline[linewidth=0.1pt](0,0)(0,1)}
\endpspicture }\\
  0&  3&  2&  0&
{\psset{unit=0.15} \pspicture[shift=*](0,0)(39,0.45)
\psframe[linewidth=0.1pt](0.0,0)(39,1)
\psframe[linewidth=0.1pt,fillstyle=solid,fillcolor=Color1](0,0)(1,1)
\psline[linewidth=0.4pt,linecolor=white](.2,0.5)(.8,0.5)
\psline[linewidth=0.4pt,linecolor=white](.5,0.2)(.5,0.8)
\psframe[linewidth=0.1pt,fillstyle=solid,fillcolor=Color2](12,0)(13,1)
\psline[linewidth=0.4pt,linecolor=white](12.2,0.5)(12.8,0.5)
\psframe[linewidth=0.1pt,fillstyle=solid,fillcolor=Color3](13,0)(14,1)
\psline[linewidth=0.4pt,linecolor=white](13.2,0.5)(13.8,0.5)
\psframe[linewidth=0.1pt,fillstyle=solid,fillcolor=Color4](25,0)(26,1)
\psline[linewidth=0.4pt,linecolor=white](25.2,0.5)(25.8,0.5)
\psline[linewidth=0.4pt,linecolor=white](25.5,0.2)(25.5,0.8)
\multips(0,0)(1,0){39}{\psline[linewidth=0.1pt](0,0)(0,1)}
\endpspicture }\\
  0&  3&  2&  1&
{\psset{unit=0.15} \pspicture[shift=*](0,0)(39,0.45)
\psframe[linewidth=0.1pt](0.0,0)(39,1)
\psframe[linewidth=0.1pt,fillstyle=solid,fillcolor=Color1](0,0)(1,1)
\psline[linewidth=0.4pt,linecolor=white](.2,0.5)(.8,0.5)
\psline[linewidth=0.4pt,linecolor=white](.5,0.2)(.5,0.8)
\psframe[linewidth=0.1pt,fillstyle=solid,fillcolor=Color2](12,0)(13,1)
\psline[linewidth=0.4pt,linecolor=white](12.2,0.5)(12.8,0.5)
\psframe[linewidth=0.1pt,fillstyle=solid,fillcolor=Color3](13,0)(14,1)
\psline[linewidth=0.4pt,linecolor=white](13.2,0.5)(13.8,0.5)
\psframe[linewidth=0.1pt,fillstyle=solid,fillcolor=Color4](25,0)(26,1)
\psline[linewidth=0.4pt,linecolor=white](25.2,0.5)(25.8,0.5)
\psline[linewidth=0.4pt,linecolor=white](25.5,0.2)(25.5,0.8)
\multips(0,0)(1,0){39}{\psline[linewidth=0.1pt](0,0)(0,1)}
\endpspicture }\\
\hline  1&  0&  0&  0&
{\psset{unit=0.15} \pspicture[shift=*](0,0)(39,0.45)
\psframe[linewidth=0.1pt](0.0,0)(39,1)
\psframe[linewidth=0.1pt,fillstyle=solid,fillcolor=Color1](0,0)(1,1)
\psline[linewidth=0.4pt,linecolor=white](.2,0.5)(.8,0.5)
\psline[linewidth=0.4pt,linecolor=white](.5,0.2)(.5,0.8)
\psframe[linewidth=0.1pt,fillstyle=solid,fillcolor=Color2](14,0)(15,1)
\psline[linewidth=0.4pt,linecolor=white](14.2,0.5)(14.8,0.5)
\psline[linewidth=0.4pt,linecolor=white](14.5,0.2)(14.5,0.8)
\psframe[linewidth=0.1pt,fillstyle=solid,fillcolor=Color3](26,0)(27,1)
\psline[linewidth=0.4pt,linecolor=white](26.2,0.5)(26.8,0.5)
\psframe[linewidth=0.1pt,fillstyle=solid,fillcolor=Color4](27,0)(28,1)
\psline[linewidth=0.4pt,linecolor=white](27.2,0.5)(27.8,0.5)
\multips(0,0)(1,0){39}{\psline[linewidth=0.1pt](0,0)(0,1)}
\endpspicture }\\
  1&  0&  0&  1&
{\psset{unit=0.15} \pspicture[shift=*](0,0)(39,0.45)
\psframe[linewidth=0.1pt](0.0,0)(39,1)
\psframe[linewidth=0.1pt,fillstyle=solid,fillcolor=Color1](0,0)(1,1)
\psline[linewidth=0.4pt,linecolor=white](.2,0.5)(.8,0.5)
\psline[linewidth=0.4pt,linecolor=white](.5,0.2)(.5,0.8)
\psframe[linewidth=0.1pt,fillstyle=solid,fillcolor=Color2](14,0)(15,1)
\psline[linewidth=0.4pt,linecolor=white](14.2,0.5)(14.8,0.5)
\psline[linewidth=0.4pt,linecolor=white](14.5,0.2)(14.5,0.8)
\psframe[linewidth=0.1pt,fillstyle=solid,fillcolor=Color3](26,0)(27,1)
\psline[linewidth=0.4pt,linecolor=white](26.2,0.5)(26.8,0.5)
\psframe[linewidth=0.1pt,fillstyle=solid,fillcolor=Color4](27,0)(28,1)
\psline[linewidth=0.4pt,linecolor=white](27.2,0.5)(27.8,0.5)
\multips(0,0)(1,0){39}{\psline[linewidth=0.1pt](0,0)(0,1)}
\endpspicture }\\
  1&  0&  1&  0&
{\psset{unit=0.15} \pspicture[shift=*](0,0)(39,0.45)
\psframe[linewidth=0.1pt](0.0,0)(39,1)
\psframe[linewidth=0.1pt,fillstyle=solid,fillcolor=Color1](0,0)(1,1)
\psline[linewidth=0.4pt,linecolor=white](.2,0.5)(.8,0.5)
\psline[linewidth=0.4pt,linecolor=white](.5,0.2)(.5,0.8)
\psframe[linewidth=0.1pt,fillstyle=solid,fillcolor=Color2](14,0)(15,1)
\psline[linewidth=0.4pt,linecolor=white](14.2,0.5)(14.8,0.5)
\psline[linewidth=0.4pt,linecolor=white](14.5,0.2)(14.5,0.8)
\psframe[linewidth=0.1pt,fillstyle=solid,fillcolor=Color3](26,0)(27,1)
\psline[linewidth=0.4pt,linecolor=white](26.2,0.5)(26.8,0.5)
\psframe[linewidth=0.1pt,fillstyle=solid,fillcolor=Color4](27,0)(28,1)
\psline[linewidth=0.4pt,linecolor=white](27.2,0.5)(27.8,0.5)
\multips(0,0)(1,0){39}{\psline[linewidth=0.1pt](0,0)(0,1)}
\endpspicture }\\
  1&  0&  1&  1&
{\psset{unit=0.15} \pspicture[shift=*](0,0)(39,0.45)
\psframe[linewidth=0.1pt](0.0,0)(39,1)
\psframe[linewidth=0.1pt,fillstyle=solid,fillcolor=Color1](0,0)(1,1)
\psline[linewidth=0.4pt,linecolor=white](.2,0.5)(.8,0.5)
\psline[linewidth=0.4pt,linecolor=white](.5,0.2)(.5,0.8)
\psframe[linewidth=0.1pt,fillstyle=solid,fillcolor=Color2](14,0)(15,1)
\psline[linewidth=0.4pt,linecolor=white](14.2,0.5)(14.8,0.5)
\psline[linewidth=0.4pt,linecolor=white](14.5,0.2)(14.5,0.8)
\psframe[linewidth=0.1pt,fillstyle=solid,fillcolor=Color3](26,0)(27,1)
\psline[linewidth=0.4pt,linecolor=white](26.2,0.5)(26.8,0.5)
\psframe[linewidth=0.1pt,fillstyle=solid,fillcolor=Color4](27,0)(28,1)
\psline[linewidth=0.4pt,linecolor=white](27.2,0.5)(27.8,0.5)
\multips(0,0)(1,0){39}{\psline[linewidth=0.1pt](0,0)(0,1)}
\endpspicture }\\
\hline  1&  0&  2&  0&
{\psset{unit=0.15} \pspicture[shift=*](0,0)(39,0.45)
\psframe[linewidth=0.1pt](0.0,0)(39,1)
\psframe[linewidth=0.1pt,fillstyle=solid,fillcolor=Color1](0,0)(1,1)
\psline[linewidth=0.4pt,linecolor=white](.2,0.5)(.8,0.5)
\psline[linewidth=0.4pt,linecolor=white](.5,0.2)(.5,0.8)
\psframe[linewidth=0.1pt,fillstyle=solid,fillcolor=Color2](15,0)(17,1)
\psline[linewidth=0.4pt,linecolor=white](15.2,0.5)(15.8,0.5)
\psline[linewidth=0.4pt,linecolor=white](16.2,0.5)(16.8,0.5)
\psframe[linewidth=0.1pt,fillstyle=solid,fillcolor=Color3](26,0)(27,1)
\psline[linewidth=0.4pt,linecolor=white](26.2,0.5)(26.8,0.5)
\psframe[linewidth=0.1pt,fillstyle=solid,fillcolor=Color4](28,0)(30,1)
\psline[linewidth=0.4pt,linecolor=white](28.2,0.5)(28.8,0.5)
\psline[linewidth=0.4pt,linecolor=white](28.5,0.2)(28.5,0.8)
\psline[linewidth=0.4pt,linecolor=white](29.2,0.5)(29.8,0.5)
\psline[linewidth=0.4pt,linecolor=white](29.5,0.2)(29.5,0.8)
\multips(0,0)(1,0){39}{\psline[linewidth=0.1pt](0,0)(0,1)}
\endpspicture }\\
  1&  0&  2&  1&
{\psset{unit=0.15} \pspicture[shift=*](0,0)(39,0.45)
\psframe[linewidth=0.1pt](0.0,0)(39,1)
\psframe[linewidth=0.1pt,fillstyle=solid,fillcolor=Color1](0,0)(1,1)
\psline[linewidth=0.4pt,linecolor=white](.2,0.5)(.8,0.5)
\psline[linewidth=0.4pt,linecolor=white](.5,0.2)(.5,0.8)
\psframe[linewidth=0.1pt,fillstyle=solid,fillcolor=Color2](15,0)(17,1)
\psline[linewidth=0.4pt,linecolor=white](15.2,0.5)(15.8,0.5)
\psline[linewidth=0.4pt,linecolor=white](16.2,0.5)(16.8,0.5)
\psframe[linewidth=0.1pt,fillstyle=solid,fillcolor=Color3](26,0)(27,1)
\psline[linewidth=0.4pt,linecolor=white](26.2,0.5)(26.8,0.5)
\psframe[linewidth=0.1pt,fillstyle=solid,fillcolor=Color4](28,0)(30,1)
\psline[linewidth=0.4pt,linecolor=white](28.2,0.5)(28.8,0.5)
\psline[linewidth=0.4pt,linecolor=white](28.5,0.2)(28.5,0.8)
\psline[linewidth=0.4pt,linecolor=white](29.2,0.5)(29.8,0.5)
\psline[linewidth=0.4pt,linecolor=white](29.5,0.2)(29.5,0.8)
\multips(0,0)(1,0){39}{\psline[linewidth=0.1pt](0,0)(0,1)}
\endpspicture }\\
\hline  1&  1&  0&  0&
{\psset{unit=0.15} \pspicture[shift=*](0,0)(39,0.45)
\psframe[linewidth=0.1pt](0.0,0)(39,1)
\psframe[linewidth=0.1pt,fillstyle=solid,fillcolor=Color1](0,0)(1,1)
\psline[linewidth=0.4pt,linecolor=white](.2,0.5)(.8,0.5)
\psline[linewidth=0.4pt,linecolor=white](.5,0.2)(.5,0.8)
\psframe[linewidth=0.1pt,fillstyle=solid,fillcolor=Color2](17,0)(18,1)
\psline[linewidth=0.4pt,linecolor=white](17.2,0.5)(17.8,0.5)
\psline[linewidth=0.4pt,linecolor=white](17.5,0.2)(17.5,0.8)
\psframe[linewidth=0.1pt,fillstyle=solid,fillcolor=Color3](26,0)(27,1)
\psline[linewidth=0.4pt,linecolor=white](26.2,0.5)(26.8,0.5)
\psframe[linewidth=0.1pt,fillstyle=solid,fillcolor=Color4](30,0)(31,1)
\psline[linewidth=0.4pt,linecolor=white](30.2,0.5)(30.8,0.5)
\multips(0,0)(1,0){39}{\psline[linewidth=0.1pt](0,0)(0,1)}
\endpspicture }\\
  1&  1&  0&  1&
{\psset{unit=0.15} \pspicture[shift=*](0,0)(39,0.45)
\psframe[linewidth=0.1pt](0.0,0)(39,1)
\psframe[linewidth=0.1pt,fillstyle=solid,fillcolor=Color1](0,0)(1,1)
\psline[linewidth=0.4pt,linecolor=white](.2,0.5)(.8,0.5)
\psline[linewidth=0.4pt,linecolor=white](.5,0.2)(.5,0.8)
\psframe[linewidth=0.1pt,fillstyle=solid,fillcolor=Color2](17,0)(18,1)
\psline[linewidth=0.4pt,linecolor=white](17.2,0.5)(17.8,0.5)
\psline[linewidth=0.4pt,linecolor=white](17.5,0.2)(17.5,0.8)
\psframe[linewidth=0.1pt,fillstyle=solid,fillcolor=Color3](26,0)(27,1)
\psline[linewidth=0.4pt,linecolor=white](26.2,0.5)(26.8,0.5)
\psframe[linewidth=0.1pt,fillstyle=solid,fillcolor=Color4](30,0)(31,1)
\psline[linewidth=0.4pt,linecolor=white](30.2,0.5)(30.8,0.5)
\multips(0,0)(1,0){39}{\psline[linewidth=0.1pt](0,0)(0,1)}
\endpspicture }\\
  1&  1&  1&  0&
{\psset{unit=0.15} \pspicture[shift=*](0,0)(39,0.45)
\psframe[linewidth=0.1pt](0.0,0)(39,1)
\psframe[linewidth=0.1pt,fillstyle=solid,fillcolor=Color1](0,0)(1,1)
\psline[linewidth=0.4pt,linecolor=white](.2,0.5)(.8,0.5)
\psline[linewidth=0.4pt,linecolor=white](.5,0.2)(.5,0.8)
\psframe[linewidth=0.1pt,fillstyle=solid,fillcolor=Color2](17,0)(18,1)
\psline[linewidth=0.4pt,linecolor=white](17.2,0.5)(17.8,0.5)
\psline[linewidth=0.4pt,linecolor=white](17.5,0.2)(17.5,0.8)
\psframe[linewidth=0.1pt,fillstyle=solid,fillcolor=Color3](26,0)(27,1)
\psline[linewidth=0.4pt,linecolor=white](26.2,0.5)(26.8,0.5)
\psframe[linewidth=0.1pt,fillstyle=solid,fillcolor=Color4](30,0)(31,1)
\psline[linewidth=0.4pt,linecolor=white](30.2,0.5)(30.8,0.5)
\multips(0,0)(1,0){39}{\psline[linewidth=0.1pt](0,0)(0,1)}
\endpspicture }\\
  1&  1&  1&  1&
{\psset{unit=0.15} \pspicture[shift=*](0,0)(39,0.45)
\psframe[linewidth=0.1pt](0.0,0)(39,1)
\psframe[linewidth=0.1pt,fillstyle=solid,fillcolor=Color1](0,0)(1,1)
\psline[linewidth=0.4pt,linecolor=white](.2,0.5)(.8,0.5)
\psline[linewidth=0.4pt,linecolor=white](.5,0.2)(.5,0.8)
\psframe[linewidth=0.1pt,fillstyle=solid,fillcolor=Color2](17,0)(18,1)
\psline[linewidth=0.4pt,linecolor=white](17.2,0.5)(17.8,0.5)
\psline[linewidth=0.4pt,linecolor=white](17.5,0.2)(17.5,0.8)
\psframe[linewidth=0.1pt,fillstyle=solid,fillcolor=Color3](26,0)(27,1)
\psline[linewidth=0.4pt,linecolor=white](26.2,0.5)(26.8,0.5)
\psframe[linewidth=0.1pt,fillstyle=solid,fillcolor=Color4](30,0)(31,1)
\psline[linewidth=0.4pt,linecolor=white](30.2,0.5)(30.8,0.5)
\multips(0,0)(1,0){39}{\psline[linewidth=0.1pt](0,0)(0,1)}
\endpspicture }\\
\hline  1&  1&  2&  0&
{\psset{unit=0.15} \pspicture[shift=*](0,0)(39,0.45)
\psframe[linewidth=0.1pt](0.0,0)(39,1)
\psframe[linewidth=0.1pt,fillstyle=solid,fillcolor=Color1](0,0)(1,1)
\psline[linewidth=0.4pt,linecolor=white](.2,0.5)(.8,0.5)
\psline[linewidth=0.4pt,linecolor=white](.5,0.2)(.5,0.8)
\psframe[linewidth=0.1pt,fillstyle=solid,fillcolor=Color2](18,0)(20,1)
\psline[linewidth=0.4pt,linecolor=white](18.2,0.5)(18.8,0.5)
\psline[linewidth=0.4pt,linecolor=white](19.2,0.5)(19.8,0.5)
\psframe[linewidth=0.1pt,fillstyle=solid,fillcolor=Color3](26,0)(27,1)
\psline[linewidth=0.4pt,linecolor=white](26.2,0.5)(26.8,0.5)
\psframe[linewidth=0.1pt,fillstyle=solid,fillcolor=Color4](31,0)(33,1)
\psline[linewidth=0.4pt,linecolor=white](31.2,0.5)(31.8,0.5)
\psline[linewidth=0.4pt,linecolor=white](31.5,0.2)(31.5,0.8)
\psline[linewidth=0.4pt,linecolor=white](32.2,0.5)(32.8,0.5)
\psline[linewidth=0.4pt,linecolor=white](32.5,0.2)(32.5,0.8)
\multips(0,0)(1,0){39}{\psline[linewidth=0.1pt](0,0)(0,1)}
\endpspicture }\\
  1&  1&  2&  1&
{\psset{unit=0.15} \pspicture[shift=*](0,0)(39,0.45)
\psframe[linewidth=0.1pt](0.0,0)(39,1)
\psframe[linewidth=0.1pt,fillstyle=solid,fillcolor=Color1](0,0)(1,1)
\psline[linewidth=0.4pt,linecolor=white](.2,0.5)(.8,0.5)
\psline[linewidth=0.4pt,linecolor=white](.5,0.2)(.5,0.8)
\psframe[linewidth=0.1pt,fillstyle=solid,fillcolor=Color2](18,0)(20,1)
\psline[linewidth=0.4pt,linecolor=white](18.2,0.5)(18.8,0.5)
\psline[linewidth=0.4pt,linecolor=white](19.2,0.5)(19.8,0.5)
\psframe[linewidth=0.1pt,fillstyle=solid,fillcolor=Color3](26,0)(27,1)
\psline[linewidth=0.4pt,linecolor=white](26.2,0.5)(26.8,0.5)
\psframe[linewidth=0.1pt,fillstyle=solid,fillcolor=Color4](31,0)(33,1)
\psline[linewidth=0.4pt,linecolor=white](31.2,0.5)(31.8,0.5)
\psline[linewidth=0.4pt,linecolor=white](31.5,0.2)(31.5,0.8)
\psline[linewidth=0.4pt,linecolor=white](32.2,0.5)(32.8,0.5)
\psline[linewidth=0.4pt,linecolor=white](32.5,0.2)(32.5,0.8)
\multips(0,0)(1,0){39}{\psline[linewidth=0.1pt](0,0)(0,1)}
\endpspicture }\\
\hline  1&  2&  0&  0&
{\psset{unit=0.15} \pspicture[shift=*](0,0)(39,0.45)
\psframe[linewidth=0.1pt](0.0,0)(39,1)
\psframe[linewidth=0.1pt,fillstyle=solid,fillcolor=Color1](0,0)(1,1)
\psline[linewidth=0.4pt,linecolor=white](.2,0.5)(.8,0.5)
\psline[linewidth=0.4pt,linecolor=white](.5,0.2)(.5,0.8)
\psframe[linewidth=0.1pt,fillstyle=solid,fillcolor=Color2](20,0)(21,1)
\psline[linewidth=0.4pt,linecolor=white](20.2,0.5)(20.8,0.5)
\psline[linewidth=0.4pt,linecolor=white](20.5,0.2)(20.5,0.8)
\psframe[linewidth=0.1pt,fillstyle=solid,fillcolor=Color3](26,0)(27,1)
\psline[linewidth=0.4pt,linecolor=white](26.2,0.5)(26.8,0.5)
\psframe[linewidth=0.1pt,fillstyle=solid,fillcolor=Color4](33,0)(34,1)
\psline[linewidth=0.4pt,linecolor=white](33.2,0.5)(33.8,0.5)
\multips(0,0)(1,0){39}{\psline[linewidth=0.1pt](0,0)(0,1)}
\endpspicture }\\
  1&  2&  0&  1&
{\psset{unit=0.15} \pspicture[shift=*](0,0)(39,0.45)
\psframe[linewidth=0.1pt](0.0,0)(39,1)
\psframe[linewidth=0.1pt,fillstyle=solid,fillcolor=Color1](0,0)(1,1)
\psline[linewidth=0.4pt,linecolor=white](.2,0.5)(.8,0.5)
\psline[linewidth=0.4pt,linecolor=white](.5,0.2)(.5,0.8)
\psframe[linewidth=0.1pt,fillstyle=solid,fillcolor=Color2](20,0)(21,1)
\psline[linewidth=0.4pt,linecolor=white](20.2,0.5)(20.8,0.5)
\psline[linewidth=0.4pt,linecolor=white](20.5,0.2)(20.5,0.8)
\psframe[linewidth=0.1pt,fillstyle=solid,fillcolor=Color3](26,0)(27,1)
\psline[linewidth=0.4pt,linecolor=white](26.2,0.5)(26.8,0.5)
\psframe[linewidth=0.1pt,fillstyle=solid,fillcolor=Color4](33,0)(34,1)
\psline[linewidth=0.4pt,linecolor=white](33.2,0.5)(33.8,0.5)
\multips(0,0)(1,0){39}{\psline[linewidth=0.1pt](0,0)(0,1)}
\endpspicture }\\
  1&  2&  1&  0&
{\psset{unit=0.15} \pspicture[shift=*](0,0)(39,0.45)
\psframe[linewidth=0.1pt](0.0,0)(39,1)
\psframe[linewidth=0.1pt,fillstyle=solid,fillcolor=Color1](0,0)(1,1)
\psline[linewidth=0.4pt,linecolor=white](.2,0.5)(.8,0.5)
\psline[linewidth=0.4pt,linecolor=white](.5,0.2)(.5,0.8)
\psframe[linewidth=0.1pt,fillstyle=solid,fillcolor=Color2](20,0)(21,1)
\psline[linewidth=0.4pt,linecolor=white](20.2,0.5)(20.8,0.5)
\psline[linewidth=0.4pt,linecolor=white](20.5,0.2)(20.5,0.8)
\psframe[linewidth=0.1pt,fillstyle=solid,fillcolor=Color3](26,0)(27,1)
\psline[linewidth=0.4pt,linecolor=white](26.2,0.5)(26.8,0.5)
\psframe[linewidth=0.1pt,fillstyle=solid,fillcolor=Color4](33,0)(34,1)
\psline[linewidth=0.4pt,linecolor=white](33.2,0.5)(33.8,0.5)
\multips(0,0)(1,0){39}{\psline[linewidth=0.1pt](0,0)(0,1)}
\endpspicture }\\
  1&  2&  1&  1&
{\psset{unit=0.15} \pspicture[shift=*](0,0)(39,0.45)
\psframe[linewidth=0.1pt](0.0,0)(39,1)
\psframe[linewidth=0.1pt,fillstyle=solid,fillcolor=Color1](0,0)(1,1)
\psline[linewidth=0.4pt,linecolor=white](.2,0.5)(.8,0.5)
\psline[linewidth=0.4pt,linecolor=white](.5,0.2)(.5,0.8)
\psframe[linewidth=0.1pt,fillstyle=solid,fillcolor=Color2](20,0)(21,1)
\psline[linewidth=0.4pt,linecolor=white](20.2,0.5)(20.8,0.5)
\psline[linewidth=0.4pt,linecolor=white](20.5,0.2)(20.5,0.8)
\psframe[linewidth=0.1pt,fillstyle=solid,fillcolor=Color3](26,0)(27,1)
\psline[linewidth=0.4pt,linecolor=white](26.2,0.5)(26.8,0.5)
\psframe[linewidth=0.1pt,fillstyle=solid,fillcolor=Color4](33,0)(34,1)
\psline[linewidth=0.4pt,linecolor=white](33.2,0.5)(33.8,0.5)
\multips(0,0)(1,0){39}{\psline[linewidth=0.1pt](0,0)(0,1)}
\endpspicture }\\
\hline  1&  2&  2&  0&
{\psset{unit=0.15} \pspicture[shift=*](0,0)(39,0.45)
\psframe[linewidth=0.1pt](0.0,0)(39,1)
\psframe[linewidth=0.1pt,fillstyle=solid,fillcolor=Color1](0,0)(1,1)
\psline[linewidth=0.4pt,linecolor=white](.2,0.5)(.8,0.5)
\psline[linewidth=0.4pt,linecolor=white](.5,0.2)(.5,0.8)
\psframe[linewidth=0.1pt,fillstyle=solid,fillcolor=Color2](21,0)(23,1)
\psline[linewidth=0.4pt,linecolor=white](21.2,0.5)(21.8,0.5)
\psline[linewidth=0.4pt,linecolor=white](22.2,0.5)(22.8,0.5)
\psframe[linewidth=0.1pt,fillstyle=solid,fillcolor=Color3](26,0)(27,1)
\psline[linewidth=0.4pt,linecolor=white](26.2,0.5)(26.8,0.5)
\psframe[linewidth=0.1pt,fillstyle=solid,fillcolor=Color4](34,0)(36,1)
\psline[linewidth=0.4pt,linecolor=white](34.2,0.5)(34.8,0.5)
\psline[linewidth=0.4pt,linecolor=white](34.5,0.2)(34.5,0.8)
\psline[linewidth=0.4pt,linecolor=white](35.2,0.5)(35.8,0.5)
\psline[linewidth=0.4pt,linecolor=white](35.5,0.2)(35.5,0.8)
\multips(0,0)(1,0){39}{\psline[linewidth=0.1pt](0,0)(0,1)}
\endpspicture }\\
  1&  2&  2&  1&
{\psset{unit=0.15} \pspicture[shift=*](0,0)(39,0.45)
\psframe[linewidth=0.1pt](0.0,0)(39,1)
\psframe[linewidth=0.1pt,fillstyle=solid,fillcolor=Color1](0,0)(1,1)
\psline[linewidth=0.4pt,linecolor=white](.2,0.5)(.8,0.5)
\psline[linewidth=0.4pt,linecolor=white](.5,0.2)(.5,0.8)
\psframe[linewidth=0.1pt,fillstyle=solid,fillcolor=Color2](21,0)(23,1)
\psline[linewidth=0.4pt,linecolor=white](21.2,0.5)(21.8,0.5)
\psline[linewidth=0.4pt,linecolor=white](22.2,0.5)(22.8,0.5)
\psframe[linewidth=0.1pt,fillstyle=solid,fillcolor=Color3](26,0)(27,1)
\psline[linewidth=0.4pt,linecolor=white](26.2,0.5)(26.8,0.5)
\psframe[linewidth=0.1pt,fillstyle=solid,fillcolor=Color4](34,0)(36,1)
\psline[linewidth=0.4pt,linecolor=white](34.2,0.5)(34.8,0.5)
\psline[linewidth=0.4pt,linecolor=white](34.5,0.2)(34.5,0.8)
\psline[linewidth=0.4pt,linecolor=white](35.2,0.5)(35.8,0.5)
\psline[linewidth=0.4pt,linecolor=white](35.5,0.2)(35.5,0.8)
\multips(0,0)(1,0){39}{\psline[linewidth=0.1pt](0,0)(0,1)}
\endpspicture }\\
\hline  1&  3&  0&  0&
{\psset{unit=0.15} \pspicture[shift=*](0,0)(39,0.45)
\psframe[linewidth=0.1pt](0.0,0)(39,1)
\psframe[linewidth=0.1pt,fillstyle=solid,fillcolor=Color1](0,0)(1,1)
\psline[linewidth=0.4pt,linecolor=white](.2,0.5)(.8,0.5)
\psline[linewidth=0.4pt,linecolor=white](.5,0.2)(.5,0.8)
\psframe[linewidth=0.1pt,fillstyle=solid,fillcolor=Color2](23,0)(24,1)
\psline[linewidth=0.4pt,linecolor=white](23.2,0.5)(23.8,0.5)
\psline[linewidth=0.4pt,linecolor=white](23.5,0.2)(23.5,0.8)
\psframe[linewidth=0.1pt,fillstyle=solid,fillcolor=Color3](26,0)(27,1)
\psline[linewidth=0.4pt,linecolor=white](26.2,0.5)(26.8,0.5)
\psframe[linewidth=0.1pt,fillstyle=solid,fillcolor=Color4](36,0)(37,1)
\psline[linewidth=0.4pt,linecolor=white](36.2,0.5)(36.8,0.5)
\multips(0,0)(1,0){39}{\psline[linewidth=0.1pt](0,0)(0,1)}
\endpspicture }\\
  1&  3&  0&  1&
{\psset{unit=0.15} \pspicture[shift=*](0,0)(39,0.45)
\psframe[linewidth=0.1pt](0.0,0)(39,1)
\psframe[linewidth=0.1pt,fillstyle=solid,fillcolor=Color1](0,0)(1,1)
\psline[linewidth=0.4pt,linecolor=white](.2,0.5)(.8,0.5)
\psline[linewidth=0.4pt,linecolor=white](.5,0.2)(.5,0.8)
\psframe[linewidth=0.1pt,fillstyle=solid,fillcolor=Color2](23,0)(24,1)
\psline[linewidth=0.4pt,linecolor=white](23.2,0.5)(23.8,0.5)
\psline[linewidth=0.4pt,linecolor=white](23.5,0.2)(23.5,0.8)
\psframe[linewidth=0.1pt,fillstyle=solid,fillcolor=Color3](26,0)(27,1)
\psline[linewidth=0.4pt,linecolor=white](26.2,0.5)(26.8,0.5)
\psframe[linewidth=0.1pt,fillstyle=solid,fillcolor=Color4](36,0)(37,1)
\psline[linewidth=0.4pt,linecolor=white](36.2,0.5)(36.8,0.5)
\multips(0,0)(1,0){39}{\psline[linewidth=0.1pt](0,0)(0,1)}
\endpspicture }\\
  1&  3&  1&  0&
{\psset{unit=0.15} \pspicture[shift=*](0,0)(39,0.45)
\psframe[linewidth=0.1pt](0.0,0)(39,1)
\psframe[linewidth=0.1pt,fillstyle=solid,fillcolor=Color1](0,0)(1,1)
\psline[linewidth=0.4pt,linecolor=white](.2,0.5)(.8,0.5)
\psline[linewidth=0.4pt,linecolor=white](.5,0.2)(.5,0.8)
\psframe[linewidth=0.1pt,fillstyle=solid,fillcolor=Color2](23,0)(24,1)
\psline[linewidth=0.4pt,linecolor=white](23.2,0.5)(23.8,0.5)
\psline[linewidth=0.4pt,linecolor=white](23.5,0.2)(23.5,0.8)
\psframe[linewidth=0.1pt,fillstyle=solid,fillcolor=Color3](26,0)(27,1)
\psline[linewidth=0.4pt,linecolor=white](26.2,0.5)(26.8,0.5)
\psframe[linewidth=0.1pt,fillstyle=solid,fillcolor=Color4](36,0)(37,1)
\psline[linewidth=0.4pt,linecolor=white](36.2,0.5)(36.8,0.5)
\multips(0,0)(1,0){39}{\psline[linewidth=0.1pt](0,0)(0,1)}
\endpspicture }\\
  1&  3&  1&  1&
{\psset{unit=0.15} \pspicture[shift=*](0,0)(39,0.45)
\psframe[linewidth=0.1pt](0.0,0)(39,1)
\psframe[linewidth=0.1pt,fillstyle=solid,fillcolor=Color1](0,0)(1,1)
\psline[linewidth=0.4pt,linecolor=white](.2,0.5)(.8,0.5)
\psline[linewidth=0.4pt,linecolor=white](.5,0.2)(.5,0.8)
\psframe[linewidth=0.1pt,fillstyle=solid,fillcolor=Color2](23,0)(24,1)
\psline[linewidth=0.4pt,linecolor=white](23.2,0.5)(23.8,0.5)
\psline[linewidth=0.4pt,linecolor=white](23.5,0.2)(23.5,0.8)
\psframe[linewidth=0.1pt,fillstyle=solid,fillcolor=Color3](26,0)(27,1)
\psline[linewidth=0.4pt,linecolor=white](26.2,0.5)(26.8,0.5)
\psframe[linewidth=0.1pt,fillstyle=solid,fillcolor=Color4](36,0)(37,1)
\psline[linewidth=0.4pt,linecolor=white](36.2,0.5)(36.8,0.5)
\multips(0,0)(1,0){39}{\psline[linewidth=0.1pt](0,0)(0,1)}
\endpspicture }\\
\hline  1&  3&  2&  0&
{\psset{unit=0.15} \pspicture[shift=*](0,0)(39,0.45)
\psframe[linewidth=0.1pt](0.0,0)(39,1)
\psframe[linewidth=0.1pt,fillstyle=solid,fillcolor=Color1](0,0)(1,1)
\psline[linewidth=0.4pt,linecolor=white](.2,0.5)(.8,0.5)
\psline[linewidth=0.4pt,linecolor=white](.5,0.2)(.5,0.8)
\psframe[linewidth=0.1pt,fillstyle=solid,fillcolor=Color2](24,0)(26,1)
\psline[linewidth=0.4pt,linecolor=white](24.2,0.5)(24.8,0.5)
\psline[linewidth=0.4pt,linecolor=white](25.2,0.5)(25.8,0.5)
\psframe[linewidth=0.1pt,fillstyle=solid,fillcolor=Color3](26,0)(27,1)
\psline[linewidth=0.4pt,linecolor=white](26.2,0.5)(26.8,0.5)
\psframe[linewidth=0.1pt,fillstyle=solid,fillcolor=Color4](37,0)(39,1)
\psline[linewidth=0.4pt,linecolor=white](37.2,0.5)(37.8,0.5)
\psline[linewidth=0.4pt,linecolor=white](37.5,0.2)(37.5,0.8)
\psline[linewidth=0.4pt,linecolor=white](38.2,0.5)(38.8,0.5)
\psline[linewidth=0.4pt,linecolor=white](38.5,0.2)(38.5,0.8)
\multips(0,0)(1,0){39}{\psline[linewidth=0.1pt](0,0)(0,1)}
\endpspicture }\\
  1&  3&  2&  1&
{\psset{unit=0.15} \pspicture[shift=*](0,0)(39,0.45)
\psframe[linewidth=0.1pt](0.0,0)(39,1)
\psframe[linewidth=0.1pt,fillstyle=solid,fillcolor=Color1](0,0)(1,1)
\psline[linewidth=0.4pt,linecolor=white](.2,0.5)(.8,0.5)
\psline[linewidth=0.4pt,linecolor=white](.5,0.2)(.5,0.8)
\psframe[linewidth=0.1pt,fillstyle=solid,fillcolor=Color2](24,0)(26,1)
\psline[linewidth=0.4pt,linecolor=white](24.2,0.5)(24.8,0.5)
\psline[linewidth=0.4pt,linecolor=white](25.2,0.5)(25.8,0.5)
\psframe[linewidth=0.1pt,fillstyle=solid,fillcolor=Color3](26,0)(27,1)
\psline[linewidth=0.4pt,linecolor=white](26.2,0.5)(26.8,0.5)
\psframe[linewidth=0.1pt,fillstyle=solid,fillcolor=Color4](37,0)(39,1)
\psline[linewidth=0.4pt,linecolor=white](37.2,0.5)(37.8,0.5)
\psline[linewidth=0.4pt,linecolor=white](37.5,0.2)(37.5,0.8)
\psline[linewidth=0.4pt,linecolor=white](38.2,0.5)(38.8,0.5)
\psline[linewidth=0.4pt,linecolor=white](38.5,0.2)(38.5,0.8)
\multips(0,0)(1,0){39}{\psline[linewidth=0.1pt](0,0)(0,1)}
\endpspicture }\\
\hline
\end{array}
\]
}

\noindent Each small box under the $f_{i}$ column stands for a term in $f_{i}%
$. The terms are ordered in the ascending order in their exponents. Thus the
first box stands for $x^{0}$ and the last box stands for $x^{p_{1}p_{2}-1}$.
Recall that $\Phi_{3\cdot13\cdot79}$ is flat, that is, the coefficients are
$-1,0,1$. An empty box stands for $0.$ Non-empty boxes have signs written on
them. The color of a box indicates which sub-polynomial of $f_{i}$ it belongs
to, as follows.
\[
f_{i}~~=~~~\underset{\color{red}S_{1}}{\underbrace{1}}%
~~~\underset{\color{green}S_{2}}{\underbrace{+g_{i}}}%
~~~\underset{\color{magenta}S_{3}}{\underbrace{-x^{\left(  i_{1}+1\right)
p_{2}}}}~~~\underset{\color{cyan}S_{4}}{\underbrace{-x^{p_{2}}\ g_{i}}}%
\]
\noindent for instance, a red box
\textcolor{red}{\rule{\fontcharht\font`X}{\fontcharht\font`X}} belongs to the
sub-polynomial $\textcolor{red}{S_1}$.

\item[\textsf{C2.}] Note that there is no overlapping of terms (no
cancellation or accumulation) and the terms are ordered in the ascending order
in their exponents, verifying Theorem \ref{thm:exp}-\textsf{C2}.
\end{enumerate}
\end{example}

\begin{remark}
A few observations.

\begin{enumerate}
\item Many $f_{i}$ are the same. For instance, in the above example, we
observe that%
\begin{align*}
f_{0000}  &  =f_{0001}\\
f_{0010}  &  =f_{0011}=f_{0020}=f_{0021}\\
f_{0100}  &  =f_{0101}\\
f_{0110}  &  =f_{0111}=f_{0120}=f_{0121}\\
&  ...
\end{align*}
We put horizontal lines to group the same ones. This pattern holds in general:%
\begin{align*}
\,\text{if}\ i_{3},i_{3}^{\ast}  &  \leq i_{1}\ \ \text{then}\ \ f_{i_{1}%
i_{2}i_{3}i_{4}}=f_{i_{1}i_{2}i_{3}^{\ast}i_{4}^{\ast}}\\
\,\text{if}\ i_{3},i_{3}^{\ast}  &  >i_{1}\ \ \text{then}\ \ f_{i_{1}%
i_{2}i_{3}i_{4}}=f_{i_{1}i_{2}i_{3}^{\ast}i_{4}^{\ast}}%
\end{align*}
It is immediate from the following two facts

\begin{enumerate}
\item The index $i_{4}$ does not appear at all in the explicit expression for
$f_{i}$.

\item The index $i_{3}$ only appears in the case selections: $i_{3}\leq
i_{1}\ $and $i_{3}>i_{1}$.
\end{enumerate}

\item By examining the ranges of the indices $i_{1}$ and $i_{2}$, one sees
immediately that there are
\[
\,2\left(  p_{1}-1\right)  q_{2}=2\left(  p_{1}-1\right)  \frac{\left(
p_{2}-1\right)  }{p_{1}}=2\frac{\varphi\left(  p_{1}p_{2}\right)  }{p_{1}}%
\]
distinct $f_{i}$'s in general. For instance, the above example has
$2\cdot\frac{\varphi(3\cdot13)}{3} = 2\cdot\frac{(3-1)(13-1)}{3} =16$ distinct
$f_{i}$'s.
\end{enumerate}
\end{remark}

\noindent As an application of the explicit expressions in
Theorems~\ref{thm:exp} we give an explicit formula for the number of nonzero
terms (hamming weight, $\mathrm{hw}$) in the cyclotomic polynomials in the family.

\begin{corollary}
\label{cor:hw} Let $p_{2}\equiv1\operatorname{mod}p_{1}$ and $p_{3}%
\equiv1\operatorname{mod}p_{1}p_{2}$. Then%
\[
\mathrm{hw}(\Phi_{p_{1}p_{2}p_{3}})=\frac{2}{3}\frac{\varphi\left(  p_{1}%
p_{2}p_{3}\right)  (p_{1}+4)}{p_{1}p_{2}}+1
\]
where $\varphi\left(  p_{1}p_{2}p_{3}\right)  =\deg\left(  \Phi_{p_{1}%
p_{2}p_{3}}\right)  =\left(  p_{1}-1\right)  \left(  p_{2}-1\right)  \left(
p_{3}-1\right)  $.
\end{corollary}

\begin{proof}
The proof is straightforward from Theorem~\ref{thm:exp}. Thus we give it here.
Note%
\begin{align*}
\mathrm{hw}\left(  \Phi_{p_{1}p_{2}p_{3}}\right)   &  =\sum_{i\in
I}\mathrm{hw}\left(  f_{i}\right)  \ \ +\mathrm{1}\\
&  =\sum_{i\in I}\left(  2+2\mathrm{hw}\left(  g_{i}\right)  \right)  \ \ +1\\
&  =\sum_{i\in I}\left\{
\begin{array}
[c]{ll}%
2(p_{1}-i_{1}) & \text{if }i_{3}\leq i_{1}\\
2(i_{1}+2) & \text{if }i_{3}>i_{1}%
\end{array}
\right.  \ \ +\ 1\\
&  =q_{3}q_{2}\sum_{0\leq i_{1}\leq p_{1}-2}\left(  2(p_{1}-i_{1})\left(
i_{1}+1\right)  +2\left(  i_{1}+2\right)  \left(  p_{1}-1-i_{1}\right)
\right)  \ \ \ +\ 1\\
&  =q_{3}q_{2}\frac{2p_{1}(p_{1}-1)\left(  p_{1}+4\right)  }{3}%
+1\ \ \,\text{(carrying out the summation)}\\
&  =\frac{2}{3}\frac{(p_{3}-1)}{p_{1}p_{2}}\left(  p_{2}-1\right)  \left(
p_{1}-1\right)  (p_{1}+4)+1\\
&  =\frac{2}{3}\frac{\varphi\left(  p_{1}p_{2}p_{3}\right)  (p_{1}+4)}%
{p_{1}p_{2}}+1
\end{align*}

\end{proof}

\begin{remark}
The above formula was derived (with a different and longer proof in \cite{AK} )
\end{remark}

\begin{example}
\label{exm:hw}\mbox{} We illustrate Corollary~\ref{cor:hw} by using two
examples (one small and one large).

\begin{enumerate}
\item Continuing from Example \ref{exm:exp}, let $p_{1}=3,\;p_{2}=13$ and
$p_{3}=79$. From Corollary \ref{cor:hw} we have
\[
{\mathrm{hw}}(\Phi_{p_{1}p_{2}p_{3}})=\frac{2}{3}\frac{(3-1)(13-1)(79-1)(3+4)}%
{\left(  3\right)  \left(  13\right)  }+1=225
\]

\item Let us consider a large example.%
\begin{align*}
p_{1}=  &  \;170141183460469231731687303715884105727=2^{127}-1\\
p_{2}=  &  \;19396094914493492417412352623610788052879\\
p_{3}=  &  \;2772062616341349718440289381107988513974840\\
&  \;91203319282999801642607689554229994773
\end{align*}

Of course $p_{1}$ is a Mersenne prime. The numbers $p_{2}$ and $p_{3}$ are the
smallest primes number such that $p_{2} \equiv1 \mod p_{1}$ and $p_{3} \equiv1
\mod {p_{1}p_{2}}$. It is practically impossible to compute $\Phi_{p_{1}%
p_{2}p_{3}}$. However one can still easily determine its hamming weight using
Corollary \ref{cor:hw}.%
\begin{align*}
{\mathrm{hw}}(\Phi_{p_{1}p_{2}p_{3}})=  &
\;31442800944722794411398673999914603816453631\\
&  \;93783142644102273813658808597364717079870210\\
&  \;3022370537039135233707348104609\\
\approx &  \;10^{118}%
\end{align*}

\end{enumerate}
\end{example}

\noindent The above corollary (Corollary~\ref{cor:hw}) in turn tells us that
the density (number of non-zeros terms\ /\ degree) is roughly inversely
proportional to $p_{2}$, when $p_{2}$ is sufficiently large, as stated in the
next corollary.

\begin{corollary}
[Density]\label{cor:density} Let $p_{2} \equiv1 \mod p_{1}$ and $p_{3} \equiv1
\mod {p_{1}p_{2}}$. Then
\[
\frac{{\mathrm{hw}}\left(  \Phi_{p_{1}p_{2}p_{3}}\right)  }{\deg\left(
\Phi_{p_{1}p_{2}p_{3}}\right)  }\approx\frac{2}{3}\frac{1}{p_{2}}%
\]
when $p_{1}\,$is sufficiently large.
\end{corollary}

\begin{proof}
The proof is immediate from Corollary~\ref{cor:hw}. Thus we give it here.
Note
\[
\frac{{\mathrm{hw}}\left(  \Phi_{p_{1}p_{2}p_{3}}\right)  }{\deg\left(
\Phi_{p_{1}p_{2}p_{3}}\right)  }=\frac{\frac{2}{3}\frac{\varphi\left(
p_{1}p_{2}p_{3}\right)  (p_{1}+4)}{p_{1}p_{2}}+1}{\varphi\left(  p_{1}%
p_{2}p_{3}\right)  }\approx\frac{\frac{2}{3}\frac{\varphi\left(  p_{1}%
p_{2}p_{3}\right)  p_{1}}{p_{1}p_{2}}}{\varphi\left(  p_{1}p_{2}p_{3}\right)
}=\frac{2}{3}\frac{1}{p_{2}}%
\]

\end{proof}

\begin{example}
For the large example above (the 2nd one in Example~\ref{exm:hw}), we have
\[
\frac{{\mathrm{hw}}(\Phi_{p_{1}p_{2}p_{3}})}{\deg(\Phi_{p_{1}p_{2}p_{3}}%
)}\approx\frac{2}{3}\frac{1}{p_{2}}\;\approx10^{-40}
\]
Thus the polynomial is extremely sparse.
\end{example}

\section{Proof}

\label{sec:proof}We will prove the main result (Theorem \ref{thm:exp}).
Assume, throughout the section, that
\[
p_{2}\equiv1\operatorname{mod}p_{1}\ \ \ \text{and\ \ \ \ }p_{3}%
\equiv1\operatorname{mod}p_{1}p_{2}%
\]

\noindent Before we plunge into the technical details, we give a bird's eye
view of the whole proof. The proof is divided into several subsections. We
explain what each subsection does.

\begin{enumerate}
\item We partition $\Phi_{p_{1}p_{2}p_{3}}$into several \textquotedblleft
blocks\textquotedblright$f_{i}$. (Lemma~\ref{lem:partition})

\item We express $f_{i}$ in terms of $\Phi_{p_{1}p_{2}}$.
(Lemma~\ref{lem:block})

\item We rewrite the known explicit expression (\ref{l=2})~for $\Phi
_{p_{1}p_{2}}$ \ so that their terms are ordered in the ascending order in
their exponents. (Lemma~\ref{lem:binary})

\item We find an explicit expression for $f_{i}$ using the expression for
$\Phi_{p_{1}p_{2}p_{3}}.$ (Lemma~\ref{lem:fi})

\item We show that the expression does not have overlapping terms and that
their terms are ordered. (Lemma~\ref{lem:no-ordered})

\item Finally we put together all the above results to prove the main result
(Theorem \ref{thm:exp})
\end{enumerate}

\subsection{Partition $\Phi_{p_{1}p_{2}p_{3}}$into $f_{i}$}

We will employ the divide-conquer-combine strategy. Specifically we will
partition $\Phi_{p_{1}p_{2}p_{3}}$ into several parts. Through numerous trial
and errors and careful analysis, we found that the following repeated
partitioning of $\Phi_{p_{1}p_{2}p_{3}}$ is elegant, enlightening and useful.

\begin{enumerate}
\item We partition $\Phi_{p_{1}p_{2}p_{3}}$ under the radix $\rho_{1}=\left(
p_{2}-1\right)  \left(  p_{3}-1\right)  $, obtaining
\[
\Phi_{p_{1}p_{2}p_{3}}=\sum_{0\leq i_{1}\leq p_{1}-2}f_{i_{1}}\ x^{i_{1}%
\rho_{1}}+x^{\varphi\left(  p_{1}p_{2}p_{3}\right)  },\ \ \ \ \ \ \ \deg
f_{i_{1}}<\rho_{1}%
\]
since
\[
\frac{\deg\Phi_{p_{1}p_{2}p_{3}}}{\rho_{1}}=\frac{\left(  p_{1}-1\right)
\left(  p_{2}-1\right)  \left(  p_{3}-1\right)  }{\left(  p_{2}-1\right)
\left(  p_{3}-1\right)  }=p_{1}-1\ \ \ \ \text{and\ \ \ \ }f_{p_{1}%
-1}=x^{\varphi\left(  p_{1}p_{2}p_{3}\right)  }%
\]

\item We partition $f_{i_{1}}$under the radix $\rho_{2}=p_{1}\left(
p_{3}-1\right)  $, obtaining
\[
f_{i_{1}}=\sum_{0\leq i_{2}\leq q_{2}-1}f_{i_{1},i_{2}}x^{i_{2}\rho_{2}%
},\ \ \ \ \ \ \ \deg f_{i_{1},i_{2}}<\rho_{2}\ \ \text{and\ }q_{2}%
={\mathrm{quo}}\left(  p_{2},p_{1}\right)
\]
since%
\[
\frac{\deg f_{i_{1}}}{\rho_{2}}<\frac{\rho_{1}}{\rho_{2}}=\frac{\left(
p_{2}-1\right)  \left(  p_{3}-1\right)  }{p_{1}\left(  p_{3}-1\right)  }%
=\frac{p_{2}-1}{p_{1}}=q_{2}%
\]

\item We partition $f_{i_{1},i_{2}}$under the radix $\rho_{3}=p_{3}-1$,
obtaining%
\[
f_{i_{1},i_{2}}=\sum_{0\leq i_{3}\leq p_{1}-1}f_{i_{1},i_{2},i_{3}}%
x^{i_{3}\rho_{3}},\ \ \ \ \ \ \ \deg f_{i_{1},i_{2},i_{3}}<\rho_{3}%
\]
since%
\[
\frac{\deg f_{i_{1},i_{2}}}{\rho_{3}}<\frac{\rho_{2}}{\rho_{3}}=\frac
{p_{1}\left(  p_{3}-1\right)  }{p_{3}-1}=p_{1}%
\]

\item We partition $f_{i_{1},i_{2},i_{3}}$ under the radix $\rho_{4}%
=p_{1}p_{2}$, obtaining%
\[
f_{i_{1},i_{2},i_{3}}=\sum_{0\leq i_{3}\leq q_{3}-1}f_{i_{1},i_{2},i_{3}%
,i_{4}}x^{i_{4}\rho_{4}},\ \ \ \ \ \ \ \deg f_{i_{1},i_{2},i_{3},i_{4}}%
<\rho_{4}\ \ \ \text{and\ \ }q_{3}={\mathrm{quo}}\left(  p_{3},p_{1}%
p_{2}\right)
\]
since $\frac{\deg f_{i_{1},i_{2},i_{3}}}{\rho_{4}}<\frac{\rho_{3}}{\rho_{4}%
}=\frac{p_{1}-1}{p_{1}p_{2}}=q_{3}$
\end{enumerate}

\noindent Put together we have%
\[
\Phi_{p_{1}p_{2}p_{3}}=\sum_{\substack{0\leq i_{1}\leq p_{1}-2\\0\leq
i_{2}\leq q_{2}-1\\0\leq i_{3}\leq p_{1}-1\\0\leq i_{4}\leq q_{3}-1}%
}f_{i_{1},i_{2},i_{3},i_{4}}x^{i_{1}\rho_{1}+i_{2}\rho_{2}+i_{3}\rho_{3}%
+i_{4}\rho_{4}}+x^{\varphi\left(  p_{1}p_{2}p_{3}\right)  }%
\text{,\ \ \ \ \ \ \ }\deg f_{i_{1},i_{2},i_{3},i_{4}}<\rho_{4}%
\]
\noindent We have proven the following lemma (written more compactly).

\begin{lemma}
[$\Phi_{p_{1}p_{2}p_{3}}$ in terms of $f_{i}$]\label{lem:partition}We have,
for some $f_{i}$,
\[
\Phi_{p_{1}p_{2}p_{3}}=\sum_{i\in I}f_{i}\ x^{i\cdot\rho}\ \ +\ \ x^{\varphi
\left(  p_{1}p_{2}p_{3}\right)  }\text{,\ \ \ \ \ \ \ \ \ }\deg f_{i}<\rho_{4}%
\]
where%
\begin{align*}
I  &  =\left\{  0,\ldots,p_{1}-2\right\}  \times\left\{  0,\ldots
,q_{2}-1\right\}  \times\left\{  0,\ldots,p_{1}-1\right\}  \times\left\{
0,\ldots,q_{3}-1\right\} \\
\rho &  =\left[  \left(  p_{2}-1\right)  \left(  p_{3}-1\right)
,\ p_{1}\left(  p_{3}-1\right)  ,\ p_{3}-1,\ p_{1}p_{2}\right]
\end{align*}

\end{lemma}

\subsection{Express $f_{i}$ in terms of $\Phi_{p_{1}p_{2}}$}

In the previous subsection, we explicitly expressed $\Phi_{p_{1}p_{2}p_{3}}$
in terms of $f_{i}$'s. Now in this subsection, we will express $f_{i}$ in
terms of $\Phi_{p_{1}p_{2}}$.

\begin{lemma}
[$f_{i}$ in terms of $\Phi_{p_{1}p_{2}}$]\label{lem:block} Let $u=i_{1}%
(p_{2}-1)+i_{2}p_{1}+i_{3}$. Then we have
\[
f_{i}=-\Psi_{p_{1}p_{2}}\cdot\mathcal{T}_{u+1}\Phi_{p_{1}p_{2}},
\]
where $\Psi_{p_{1}p_{2}}\left(  x\right)  =\frac{x^{p_{1}p_{2}}-1}{\Phi
_{p_{1}p_{2}}\left(  x\right)  }$ and$\ \ \mathcal{T}_{s}\left(  \cdot\right)
={\mathrm{rem}}(\cdot,x^{s})$
\end{lemma}

\begin{proof}
Note
\begin{align*}
\Phi_{p_{1}p_{2}p_{3}}  &  =\frac{\Phi_{p_{1}p_{2}}\left(  x^{p_{3}}\right)
}{\Phi_{p_{1}p_{2}}} &  & \\
&  =-\ \Phi_{p_{1}p_{2}}\left(  x^{p_{3}}\right)  \ \Psi_{p_{1}p_{2}}%
\ \frac{1}{1-x^{p_{1}p_{2}}} &  & \\
&  =-\ \Phi_{p_{1}p_{2}}\left(  x^{p_{3}}\right)  \ \Psi_{p_{1}p_{2}}%
\ \sum_{t\geq0}x^{tp_{1}p_{2}} &  &  \text{by carrying out a formal expansion
of }\frac{1}{1-x^{p_{1}p_{2}}}\\
&  =-\ \sum_{s\geq0}a_{s}x^{sp_{3}}\ \Psi_{p_{1}p_{2}}\ \sum_{t\geq0}%
x^{tp_{1}p_{2}} &  &  \text{where }\Phi_{p_{1}p_{2}}\left(  x\right)
=\sum_{s\geq0}a_{s}x^{s}\text{, \ }a_{s}=0\ \text{for }s>\varphi\left(
p_{1}p_{2}\right) \\
&  =-\ \sum_{s\geq0}a_{s}x^{sq_{3}p_{1}p_{2}}x^{s}\ \Psi_{p_{1}p_{2}}%
\ \sum_{t\geq0}x^{tp_{1}p_{2}} &  &  \text{since }p_{3}=q_{3}p_{1}p_{2}+1\\
&  =-\ \sum_{s\geq0}a_{s}x^{sq_{3}p_{1}p_{2}}x^{s}\ \Psi_{p_{1}p_{2}}%
\sum_{u\geq0}\sum_{v=0}^{q_{3}-1}x^{(uq_{3}+v)p_{1}p_{2}} &  &  \text{where
}u={\mathrm{quo}}(t,q_{3}),~v={\mathrm{rem}}(t,q_{3})\\
&  =-\ \sum_{v=0}^{q_{3}-1}\Psi_{p_{1}p_{2}}\sum_{u\geq0,\ s\geq0}a_{s}%
x^{s}\ x^{(u+s)q_{3}p_{1}p_{2}+vp_{1}p_{2}} &  &  \text{by reordering and
combining }\\
&  =-\ \sum_{v=0}^{q_{3}-1}\Psi_{p_{1}p_{2}}\sum_{u\geq s,\ s\geq0\ }%
a_{s}x^{s}\ x^{uq_{3}p_{1}p_{2}+vp_{1}p_{2}} &  &  \text{by reindexing
}u+s\ \text{with }u\\
&  =-\ \sum_{v=0}^{q_{3}-1}\Psi_{p_{1}p_{2}}\sum_{u\geq0}\sum_{s=0}^{u}%
a_{s}x^{s}\ x^{uq_{3}p_{1}p_{2}+vp_{1}p_{2}} &  &  \text{by rewriting }%
\sum_{u\geq s,\ s\geq0\ }\ \text{into an iterated sum}\\
&  =\sum_{u\geq0}\sum_{v=0}^{q_{3}-1}\left(  -\ \Psi_{p_{1}p_{2}}\sum
_{s=0}^{u}a_{s}x^{s}\right)  \ x^{uq_{3}p_{1}p_{2}+vp_{1}p_{2}} &  &  \text{by
reordering}%
\end{align*}
Let
\[
h_{u}=-\Psi_{p_{1}p_{2}}\sum_{s=0}^{u}a_{s}x^{s}%
\]
For $u\geq\varphi(p_{1}p_{2}),$ we have
\[
h_{u}=-\Psi_{p_{1}p_{2}}~\sum_{s=0}^{\varphi(p_{1}p_{2})}a_{s}x^{s}%
=-\Psi_{p_{1}p_{2}}\Phi_{p_{1}p_{2}}=1-x^{p_{1}p_{2}}%
\]
Therefore we have
\begin{align*}
\Phi_{p_{1}p_{2}p_{3}}  &  =-~\sum_{u=0}^{\varphi(p_{1}p_{2})-1}\sum
_{v=0}^{q_{3}-1}h_{u}x^{uq_{3}p_{1}p_{2}+vp_{1}p_{2}}+\sum_{u\geq\varphi
(p_{1}p_{2})}\sum_{v=0}^{q_{3}-1}(1-x^{p_{1}p_{2}})x^{uq_{3}p_{1}p_{2}%
+vp_{1}p_{2}} &  & \\
&  =-~\sum_{u=0}^{\varphi(p_{1}p_{2})-1}\sum_{v=0}^{q_{3}-1}h_{u}%
x^{uq_{3}p_{1}p_{2}+vp_{1}p_{2}}+\sum_{u\geq\varphi(p_{1}p_{2})}x^{uq_{3}%
p_{1}p_{2}}\sum_{v=0}^{q_{3}-1}\left(  x^{vp_{1}p_{2}}-x^{\left(  v+1\right)
p_{1}p_{2}}\right)  &  &  \text{by rearranging}\\
&  =-~\sum_{u=0}^{\varphi(p_{1}p_{2})-1}\sum_{v=0}^{q_{3}-1}h_{u}%
x^{uq_{3}p_{1}p_{2}+vp_{1}p_{2}}+\sum_{u\geq\varphi(p_{1}p_{2})}x^{uq_{3}%
p_{1}p_{2}}(1-x^{q_{3}p_{1}p_{2}}) &  &  \text{by telescoping sum}\\
&  =-~\sum_{u=0}^{\varphi(p_{1}p_{2})-1}\sum_{v=0}^{q_{3}-1}h_{u}%
x^{uq_{3}p_{1}p_{2}+vp_{1}p_{2}}+\sum_{u\geq\varphi(p_{1}p_{2})}\left(
x^{uq_{3}p_{1}p_{2}}-x^{(u+1)q_{3}p_{1}p_{2}}\right)  &  &  \text{by
rearranging}\\
&  =-~\sum_{u=0}^{\varphi(p_{1}p_{2})-1}\sum_{v=0}^{q_{3}-1}h_{u}%
x^{uq_{3}p_{1}p_{2}+vp_{1}p_{2}}+x^{\varphi(p_{1}p_{2})q_{3}p_{1}p_{2}} &  &
\text{by telescoping sum}\\
&  =-~\sum_{u=0}^{\varphi(p_{1}p_{2})-1}\sum_{v=0}^{q_{3}-1}h_{u}x^{u\left(
p_{3}-1\right)  +vp_{1}p_{2}}+x^{\varphi(p_{1}p_{2}p_{3})} &  &  \text{since
}q_{3}p_{1}p_{2}=p_{3}-1
\end{align*}

\noindent Since $0\leq u\leq$ $\varphi(p_{1}p_{2})-1$, we can write $u$ as
\[
u=i_{1}(p_{2}-1)+i_{2}p_{1}+i_{3}%
\]
such that $0\leq i_{1}\leq p_{1}-2,\ 0\leq i_{2}\leq q_{2}-1\ $and $\ 0\leq
i_{3}\leq p_{1}-1$. Let us also rename $v$ as $i_{4}$. Then we have%
\begin{align*}
\Phi_{p_{1}p_{2}p_{3}}  &  =\sum_{\substack{0\leq i_{1}\leq p_{1}-2\\0\leq
i_{2}\leq q_{2}-1\\0\leq i_{3}\leq p_{1}-1\\0\leq i_{4}\leq q_{3}-1}%
}h_{u}\ x^{\left(  i_{1}(p_{2}-1)+i_{2}p_{1}+i_{3}\right)  \left(
p_{3}-1\right)  +i_{4}p_{1}p_{2}}+x^{\varphi\left(  p_{1}p_{2}p_{3}\right)  }
&  & \\
&  =\sum_{\substack{0\leq i_{1}\leq p_{1}-2\\0\leq i_{2}\leq q_{2}-1\\0\leq
i_{3}\leq p_{1}-1\\0\leq i_{4}\leq q_{3}-1}}h_{u}\ x^{i_{1}(p_{2}-1)\left(
p_{3}-1\right)  +i_{2}p_{1}\left(  p_{3}-1\right)  +i_{3}\left(
p_{3}-1\right)  +i_{4}p_{1}p_{2}}+x^{\varphi\left(  p_{1}p_{2}p_{3}\right)  }
&  &  \text{by expanding}\\
&  =\sum_{i\in I}h_{u}\ x^{i\cdot\rho}\ \ +\ \ x^{\varphi\left(  p_{1}%
p_{2}p_{3}\right)  }+x^{\varphi\left(  p_{1}p_{2}p_{3}\right)  } &  &
\text{by recalling }\rho\ \text{and }I
\end{align*}
Note that, for all $u$ in $0\leq u\leq\varphi(p_{1}p_{2})-1$, we have
\[
\deg(h_{u})\leq\deg(\Psi_{p_{1}p_{2}})+u<p_{1}p_{2}=\rho_{4}%
\]
Thus $h_{u}=f_{i}$. Hence, we have%
\[
f_{i}=-\Psi_{p_{1}p_{2}}\sum_{s=0}^{u}a_{s}x^{s}=-\Psi_{p_{1}p_{2}}%
\mathcal{T}_{u+1}\Phi_{p_{1}p_{2}}%
\]

\end{proof}

\subsection{Find an ordered explicit expression for $\Phi_{p_{1}p_{2}}$}

In the previous subsection, we expressed $f_{i}$'s in terms of $\Phi
_{p_{1}p_{2}}$. In this subsection, we find an ordered explicit expression for
$\Phi_{p_{1}p_{2}}$.

\begin{lemma}
[Ordered explicit expression for $\Phi_{p_{1}p_{2}}$]\label{lem:binary}%
\[
\Phi_{p_{1}p_{2}}=1+\sum_{a=0}^{p_{1}-2}\sum_{b=0}^{q_{2}-1}\left(
-x^{a(p_{2}-1)+bp_{1}+a+1}+x^{a(p_{2}-1)+(b+1)p_{1}}\right)
\]

\end{lemma}

\begin{proof}
We will use the known explicit expression (\ref{l=2}) for $\Phi_{p_{1}p_{2}}$
given in the introduction. From the condition $p_{2}\equiv1\operatorname{mod}%
p_{1}$, we have $s_{1}=p_{1}^{-1}\operatorname{mod}p_{2}=p_{2}-q_{2}$ and
$s_{2}=p_{2}^{-1}\operatorname{mod}p_{1}=1$. Thus, from the explicit formula
(\ref{l=2}) of $\Phi_{p_{1}p_{2}}$, we have
\begin{align*}
\Phi_{p_{1}p_{2}} &  =\sum_{\substack{0\leq b<p_{2}-q_{2}\\0\leq
a<1}}x^{bp_{1}+ap_{2}}\;\;-\;\;\sum_{\substack{0\leq b<q_{2}\\0\leq a<p_{1}%
-1}}x^{bp_{1}+ap_{2}+1} &  & \\
&  =\sum_{b=0}^{q_{2}(p_{1}-1)}x^{bp_{1}}\;\;-\;\;\sum_{a=0}^{p_{1}-2}%
\sum_{b=0}^{q_{2}-1}x^{ap_{2}+bp_{1}+1} &  &  \text{since }p_{2}-q_{2}%
=q_{2}(p_{1}-1)+1\\
&  =1+\sum_{b=1}^{q_{2}(p_{1}-1)}x^{bp_{1}}\;\;-\;\;\sum_{a=0}^{p_{1}-2}%
\sum_{b=0}^{q_{2}-1}x^{ap_{2}+bp_{1}+1} &  &  \text{by separating out the
first }b=0\\
&  =1+\sum_{a=0}^{p_{1}-2}\sum_{b=0}^{q_{2}-1}x^{(aq_{2}+b+1)p_{1}%
}\;\;-\;\;\sum_{a=0}^{p_{1}-2}\sum_{b=0}^{q_{2}-1}x^{ap_{2}+bp_{1}+1} &  &
\text{by reindexing the first }b\ \text{with }aq_{2}+b+1\\
&  =1+\sum_{a=0}^{p_{1}-2}\sum_{b=0}^{q_{2}-1}\left(  -x^{ap_{2}+bp_{1}%
+1}+x^{(aq_{2}+b+1)p_{1}}\right)   &  &  \text{by combining the sums and
reordering}\\
&  =1+\sum_{a=0}^{p_{1}-2}\sum_{b=0}^{q_{2}-1}\left(  -x^{a(p_{2}%
-1)+bp_{1}+a+1}+x^{a(p_{2}-1)+(b+1)p_{1}}\right)   &  &  \text{since }%
q_{2}p_{1}=p_{2}-1
\end{align*}
Note the terms are ordered in the ascending order of their exponents.
\end{proof}

\subsection{Find an explicit expression for $f_{i}$}

Combing the results from the previous two subsections, in this subsection, we
find an explicit expression for~$f_{i}$.

\begin{lemma}
\label{lem:trunation} Let $u=i_{1}(p_{2}-1)+i_{2}p_{1}+i_{3}$. We have
\[
\mathcal{T}_{u+1}\Phi_{p_{1}p_{2}}=1+\sum_{a=0}^{i_{1}}x^{ap_{2}}%
\sum_{\substack{0\leq b\leq q_{2}-1\;\;\;\;\;\text{if }a<i_{1}\\0\leq b\leq
i_{2}-1\;\;\;\;\;\;\text{if }a=i_{1}}}\ x^{bp_{1}}\left(  -x+x^{p_{1}%
-a}\right)  -x^{i_{1}p_{2}+i_{2}p_{1}}\mathcal{T}_{i_{3}-i_{1}+1}\left(
x\right)
\]

\end{lemma}

\begin{proof}
Note the followings.

\begin{itemize}
\item[If ] $0\leq a<i_{1},~0\leq b\leq q_{2}-1$ then we have
\begin{align*}
u+1  &  >i_{1}(p_{2}-1) &  &  \text{since }i_{2},i_{3}\geq0\\
&  \geq(i_{1}-1)(p_{2}-1)+q_{2}p_{1} &  &  \text{since }p_{2}-1=q_{2}p_{1}\\
&  \geq a(p_{2}-1)+(b+1)p_{1} &  & \\
u+1  &  >i_{1}(p_{2}-1)-p_{1}+i_{1} &  &  \text{since }i_{1}<p_{1}\\
&  =(i_{1}-1)(p_{2}-1)+(q_{2}-1)p_{1}+i_{1} &  &  \text{since }p_{2}%
-1=q_{2}p_{1}\\
&  \geq a(p_{2}-1)+bp_{1}+a+1 &  &
\end{align*}

%\begin{align*}
%a(p_{2}-1)+(b+1)p_{1}   &  \leq (i_{1}-1)(p_{2}-1)+q_{2}p_{1} \\
%& = i_{1}(p_{2}-1) &&\text{from }p_{2}-1=q_{2}p_{1}\\
%& < \iota+1  &&\text{from }  i_{2},i_{3} \geq 0\\
%a(p_{2}-1)+bp_{1}+a+1 & \leq (i_{1}-1)(p_{2}-1)+(q_{2}-1)p_{1} +i_{1} \\
%&=i_{1}(p_{2}-1)-p_{1}+i_{1}&&\text{from }p_{2}-1=q_{2}p_{1}\\
%& < \iota+1  &&\text{from }  i_{1}< p_{1}
%\end{align*}

\item[If] $a=i_{1},~0\leq b\leq i_{2}-1$ then we have
\begin{align*}
u+1  &  >i_{1}(p_{2}-1)+i_{2}p_{1} &  &  \text{since }i_{3}+1>0\\
&  \geq a(p_{2}-1)+(b+1)p_{1} &  & \\
u+1  &  >i_{1}(p_{2}-1)+(i_{2}-1)p_{1}+i_{1}+1 &  &  \text{since }i_{1}%
<p_{1}\\
&  \geq a(p_{2}-1)+bp_{1}+a+1 &  &
\end{align*}

%\begin{align*}
%a(p_{2}-1)+(b+1)p_{1} & \leq i_{1}(p_{2}-1)+i_{2}p_{1}  \\
%& < \iota+1    &&\text{from }  i_{3}+1 > 0 \\
%a(p_{2}-1)+bp_{1}+a+1 &\leq i_{1}(p_{2}-1)+(i_{2}-1)p_{1} +i_{1}+1\\
%& = i_{1}(p_{2}-1)+i_{2}p_{1}-p_{1}+i_{1}+1 \\
%& < \iota+1    &&\text{from }  i_{1}< p_{1}
%\end{align*}

\end{itemize}

\noindent Thus we have
\begin{align*}
\mathcal{T}_{u+1}\Phi_{p_{1}p_{2}} &  =\mathcal{T}_{u+1}\left(  1+\sum
_{a=0}^{p_{1}-2}\sum_{b=0}^{q_{2}-1}\left(  -x^{a(p_{2}-1)+bp_{1}%
+a+1}+x^{a(p_{2}-1)+(b+1)p_{1}}\right)  \right)   &  &  \text{from
Lemma}\ \ref{lem:binary}\\
&  =1+\sum_{a=0}^{i_{1}}\sum_{\substack{~0\leq b\leq q_{2}-1\;\;\text{if
}a<i_{1}\\~0\leq b\leq i_{2}-1\;\;\text{if }a=i_{1}}}\ \left(  -x^{a(p_{2}%
-1)+bp_{1}+a+1}+x^{a(p_{2}-1)+(b+1)p_{1}}\right)   &  & \\
&  ~~~+\mathcal{T}_{u+1}\left(  -x^{i_{1}p_{2}+i_{2}p_{1}+1}+x^{i_{1}\left(
p_{2}-1\right)  +i_{2}p_{1}+p_{1}}\right)   &  & \\
&  =1~~+~~\sum_{a=0}^{i_{1}}x^{ap_{2}}\sum_{\substack{0\leq b\leq
q_{2}-1\;\;\text{if }a<i_{1}\\0\leq b\leq i_{2}-1\;\;\text{if }a=i_{1}%
}}\ x^{bp_{1}}\left(  -x+x^{p_{1}-a}\right)   &  & \\
&  ~~~-x^{i_{1}p_{2}+i_{2}p_{1}}\mathcal{T}_{i_{3}-i_{1}+1}\left(  x\right)
&  &  \text{since }i_{1}\left(  p_{2}-1\right)  +i_{2}p_{1}+p_{1}>u\\
&  &  &  \text{and }i_{3}<p_{1}%
\end{align*}

\end{proof}

\begin{lemma}
\label{lem:pt} We have
\[
\Phi_{p_{1}}\mathcal{T}_{u+1}\Phi_{p_{1}p_{2}}=\sum_{a=0}^{i_{1}}x^{ap_{2}%
}+x^{i_{1}p_{2}}x^{i_{2}p_{1}}\sum_{k=1}^{p_{1}-1-i_{1}}x^{k}+x^{i_{1}p_{2}%
}x^{i_{2}p_{1}}\mathcal{T}_{i_{3}-i_{1}+1}\left(  -x\right)  \sum_{k=0}%
^{p_{1}-1}x^{k}%
\]

\end{lemma}

\begin{proof}
%(1-x^{q_2p_1})+x^{q_2p_1}\sum_{k=0}^{p_1-1}x^k\right)\\
%&=1 & \text{since} ~~i_{2}<q_2
%\end{align*}

%\item $i_{1}>0$\\%
\begin{align*}
\Phi_{p_{1}}\mathcal{T}_{u+1}\Phi_{p_{1}p_{2}}  &  =\sum_{k=0}^{p_{1}-1}%
x^{k}\left(  1+\sum_{a=0}^{i_{1}}x^{ap_{2}}\sum_{\substack{0\leq b\leq
q_{2}-1\;\;\;\;\;\text{if }a<i_{1}\\0\leq b\leq i_{2}-1\;\;\;\;\;\;\text{if
}a=i_{1}}}\ x^{bp_{1}}\left(  -x+x^{p_{1}-a}\right)  \right)  &  &  \text{from
Lemma \ref{lem:trunation} }\\
&  \;\;\;\;\;-\sum_{k=0}^{p_{1}-1}x^{k}x^{i_{1}p_{2}}x^{i_{2}p_{1}}%
\mathcal{T}_{i_{3}-i_{1}+1}\left(  x\right)  &  &  \text{{}}\\
&  =\underset{C}{\underbrace{\sum_{a=0}^{i_{1}}x^{ap_{2}}%
\underset{B}{\underbrace{\sum_{\substack{0\leq b\leq q_{2}-1\;\;\;\;\;\text{if
}a<i_{1}\\0\leq b\leq i_{2}-1\;\;\;\;\;\text{if }a=i_{1}}}x^{bp_{1}%
}\underset{A}{\underbrace{\left(  -x+x^{p_{1}-a}\right)  \sum_{k=0}^{p_{1}%
-1}x^{k}}}}}}} &  &  \text{by rearranging}\\
&  \;\;\;\;\;+\sum_{k=0}^{p_{1}-1}x^{k}-x^{i_{1}p_{2}}x^{i_{2}p_{1}%
}\mathcal{T}_{i_{3}-i_{1}-1}\left(  x\right)  \sum_{k=0}^{p_{1}-1}x^{k} &  &
\text{{}}%
\end{align*}
From now on, we will simplify $C$ by identifying and removing cancellable
terms in subexpressions, starting from $A.$

\begin{enumerate}
\item We simplify $A.$ Note%
\[
A=\left(  -x+x^{p_{1}-a}\right)  \sum_{k=0}^{p_{1}-1}x^{k}=-\sum_{k=0}%
^{p_{1}-1}x^{k+1}+\sum_{k=0}^{p_{1}-1}x^{p_{1}-a+k}=\left(  -1+x^{p_{1}%
}\right)  \sum_{k=1}^{p_{1}-1-a}x^{k}%
\]

\item We simplify $B.$ Note%
\begin{align*}
B  &  =\sum_{\substack{0\leq b\leq q_{2}-1\;\;\;\;\;\text{if }a<i_{1}\\0\leq
b\leq i_{2}-1\;\;\;\;\;\text{if }a=i_{1}}}x^{bp_{1}}A &  &  \text{}\\
&  =\sum_{\substack{0\leq b\leq q_{2}-1\;\;\;\;\;\text{if }a<i_{1}\\0\leq
b\leq i_{2}-1\;\;\;\;\;\text{if }a=i_{1}}}x^{bp_{1}}\left(  -1+x^{p_{1}%
}\right)  \sum_{k=1}^{p_{1}-1-a}x^{k} &  &  \text{from the expression of } A\\
&  =\sum_{\substack{0\leq b\leq q_{2}-1\;\;\;\;\;\text{if }a<i_{1}\\0\leq
b\leq i_{2}-1\;\;\;\;\;\text{if }a=i_{1}}}\left(  -x^{bp_{1}}+x^{\left(
b+1\right)  p_{1}}\right)  \sum_{k=1}^{p_{1}-1-a}x^{k} &  &  \text{}\\
&  =-\left(
\begin{array}
[c]{ll}%
1-x^{q_{2}p_{1}} & \text{if }a<i_{1}\\
1-x^{i_{2}p_{1}} & \text{if }a=i_{1}%
\end{array}
\right)  \sum_{k=1}^{p_{1}-1-a}x^{k} &  &  \text{ by telescoping sum}%
\end{align*}

\item We simplify $C.$ Note%
\begin{align*}
C  &  =\sum_{a=0}^{i_{1}}x^{ap_{2}}B &  &  \text{{}}\\
&  =-\sum_{a=0}^{i_{1}}x^{ap_{2}}\left(
\begin{array}
[c]{ll}%
1-x^{q_{2}p_{1}} & \text{if }a<i_{1}\\
1-x^{i_{2}p_{1}} & \text{if }a=i_{1}%
\end{array}
\right)  \sum_{k=1}^{p_{1}-1-a}x^{k}\ \  &  &  \text{from the expression of
}B\\
&  =-\sum_{a=0}^{i_{1}-1}x^{ap_{2}}\left(  1-x^{q_{2}p_{1}}\right)  \sum
_{k=1}^{p_{1}-1-a}x^{k} &  & \\
&  ~~~~~-x^{i_{1}p_{2}}\left(  1-x^{i_{2}p_{1}}\right)  \sum_{k=1}%
^{p_{1}-1-i_{1}}x^{k} &  &  \text{by rearranging}\\
&  =-\sum_{a=0}^{i_{1}-1}x^{ap_{2}}\sum_{k=1}^{p_{1}-1-a}x^{k}~~~+~~~\sum
_{a=0}^{i_{1}-1}x^{ap_{2}}x^{q_{2}p_{1}}\sum_{k=1}^{p_{1}-1-a}x^{k}\ \  &  &
\\
&  ~~~~~-x^{i_{1}p_{2}}\sum_{k=1}^{p_{1}-1-i_{1}}x^{k}~~~~~+~~~x^{i_{1}p_{2}%
}x^{i_{2}p_{1}}\sum_{k=1}^{p_{1}-1-i_{1}}x^{k} &  & \\
&  =-\sum_{a=0}^{i_{1}-1}x^{ap_{2}}\sum_{k=1}^{p_{1}-1-a}x^{k}~~~+~~~\sum
_{a=0}^{i_{1}-1}x^{(a+1)p_{2}-1}\sum_{k=1}^{p_{1}-1-a}x^{k} &  &  \text{since
}q_{2}p_{1}=p_{2}-1\\
&  ~~~~~-x^{i_{1}p_{2}}\sum_{k=1}^{p_{1}-1-i_{1}}x^{k}~~~~~+~~~x^{i_{1}p_{2}%
}x^{i_{2}p_{1}}\sum_{k=1}^{p_{1}-1-i_{1}}x^{k} &  & \\
&  =-\sum_{a=0}^{i_{1}-1}x^{ap_{2}}\sum_{k=1}^{p_{1}-1-a}x^{k}~~~+~~~\sum
_{a=1}^{i_{1}}x^{ap_{2}}\sum_{k=0}^{p_{1}-1-a}x^{k} &  &  \text{by reindexing
}a+1\text{\ with }a,\ \text{and }~k-1\text{\ with }k\\
&  ~~~~~-x^{i_{1}p_{2}}\sum_{k=1}^{p_{1}-1-i_{1}}x^{k}\ ~~~~+~~~x^{i_{1}p_{2}%
}x^{i_{2}p_{1}}\sum_{k=1}^{p_{1}-1-i_{1}}x^{k} &  & \\
&  =-\sum_{k=1}^{p_{1}-1}x^{k}\ ~~~+~~~\sum_{a=1}^{i_{1}}x^{ap_{2}%
}~~~+~~~x^{i_{1}p_{2}}x^{i_{2}p_{1}}\sum_{k=1}^{p_{1}-1-i_{1}}x^{k} &  &
\text{by telescoping sum}%
\end{align*}

\end{enumerate}

\noindent Thus
\begin{align*}
\Phi_{p_{1}}\mathcal{T}_{u+1}\Phi_{p_{1}p_{2}}  &  =-\sum_{k=1}^{p_{1}-1}%
x^{k}\ +\sum_{a=1}^{i_{1}}x^{ap_{2}}+x^{i_{1}p_{2}}x^{i_{2}p_{1}}\sum
_{k=1}^{p_{1}-1-i_{1}}x^{k} &  &  \text{from the expression}\\
&  ~~~~~+\sum_{k=0}^{p_{1}-1}x^{k}-x^{i_{1}p_{2}}x^{i_{2}p_{1}}\mathcal{T}%
_{i_{3}-i_{1}-1}\left(  x\right)  \sum_{k=0}^{p_{1}-1}x^{k} &  &  \text{in the
beginning of the proof}\\
&  =\sum_{a=0}^{i_{1}}x^{ap_{2}}~+~x^{i_{1}p_{2}}x^{i_{2}p_{1}}\sum
_{k=1}^{p_{1}-1-i_{1}}x^{k}~-~x^{i_{1}p_{2}}x^{i_{2}p_{1}}\mathcal{T}%
_{i_{3}-i_{1}+1}\left(  x\right)  \sum_{k=0}^{p_{1}-1}x^{k} &  &  \text{by
telescoping sum}%
\end{align*}

\end{proof}

\begin{lemma}
\label{lem:fi} We have%
\[
f_{i}=1+g_{i}-x^{\left(  i_{1}+1\right)  p_{2}}-x^{p_{2}}\ g_{i}%
\]
where%
\[
g_{i}=\left\{
\begin{array}
[c]{llll}%
\displaystyle+x^{i_{1}p_{2}+i_{2}p_{1}+1}\;\;\;\;\sum_{k=0}^{p_{1}-2-i_{1}} &
x^{k} & \text{if } & i_{3}\leq i_{1}\\
\displaystyle-x^{i_{1}\left(  p_{2}-1\right)  +\left(  i_{2}+1\right)  p_{1}%
}\sum_{k=0}^{i_{1}} & x^{k} & \text{if} & i_{3}>i_{1}%
\end{array}
\right.
\]

\end{lemma}

\begin{proof}
From Lemma \ref{lem:block} and $\Psi_{p_{1}p_{2}}=\left(  x^{p_{2}}-1\right)
\Phi_{p_{1}}$ we have:
\begin{align*}
f_{i}  &  =-\Psi_{p_{1}p_{2}}\ \mathcal{T}_{u+1}\Phi_{p_{1}p_{2}}\ \  &  &
\text{{}}\\
&  =-\left(  x^{p_{2}}-1\right)  \Phi_{p_{1}}\ \mathcal{T}_{u+1}\Phi
_{p_{1}p_{2}} &  & \\
&  =-\left(  x^{p_{2}}-1\right)  \left(  \sum_{a=0}^{i_{1}}x^{ap_{2}%
}~+~x^{i_{1}p_{2}}x^{i_{2}p_{1}}\sum_{k=1}^{p_{1}-1-i_{1}}x^{k}~-~x^{i_{1}%
p_{2}}x^{i_{2}p_{1}}\mathcal{T}_{i_{3}-i_{1}+1}\left(  x\right)  \sum
_{k=0}^{p_{1}-1}x^{k}\right)  &  &  \text{from Lemma \ref{lem:pt}}\\
&  =-\left(  x^{p_{2}}-1\right)  \sum_{a=0}^{i_{1}}x^{ap_{2}}\ \ -\left(
x^{p_{2}}-1\right)  x^{i_{1}p_{2}+i_{2}p_{1}}\sum_{k=1}^{p_{1}-1-i_{1}}%
x^{k}\ \ \ \ +\left(  x^{p_{2}}-1\right)  x^{i_{1}p_{2}+i_{2}p_{1}}%
\mathcal{T}_{i_{3}-i_{1}+1}\left(  x\right)  \sum_{k=0}^{p_{1}-1}x^{k} &  &
\text{{}}\\
&  =1\ \ -x^{\left(  i_{1}+1\right)  p_{2}}\ \ -\left(  x^{p_{2}}-1\right)
x^{i_{1}p_{2}+i_{2}p_{1}}\sum_{k=1}^{p_{1}-1-i_{1}}x^{k}\ +\left(  x^{p_{2}%
}-1\right)  x^{i_{1}p_{2}+i_{2}p_{1}}\mathcal{T}_{i_{3}-i_{1}+1}\left(
x\right)  \sum_{k=0}^{p_{1}-1}x^{k} &  &  \text{by telescoping }\\
&  &  &  \text{sum}%
\end{align*}

\noindent Now we will carry out case studies. Recall $u=i_{1}(p_{2}-1)+i_{2}p_{1}%
+i_{3}=i_{1}p_{2}+i_{2}p_{1}+i_{3}-i_{1}$

\begin{enumerate}
\item Case: $i_{3}\leq i_{1}.\ \ $Note $\mathcal{T}_{i_{3}-i_{1}+1}\left(
x\right)  =0.$ Hence%
\begin{align*}
f_{i}  &  =1\ \ \ -x^{\left(  i_{1}+1\right)  p_{2}}\ \ \ -x^{i_{1}p_{2}%
}\left(  x^{p_{2}}-1\right)  x^{i_{2}p_{1}}\sum_{k=1}^{p_{1}-i_{1}-1}x^{k} &
&  \text{{}}\\
&  =1\ \ \ -x^{\left(  i_{1}+1\right)  p_{2}}\ \ \ -x^{\left(  i_{1}+1\right)
p_{2}}x^{i_{2}p_{1}}\sum_{k=1}^{p_{1}-i_{1}-1}x^{k}\ \ \ +x^{i_{1}p_{2}%
}x^{i_{2}p_{1}}\sum_{k=1}^{p_{1}-i_{1}-1}x^{k} &  &  \text{by distribution}\\
&  =1\ \ \ +x^{i_{1}p_{2}}x^{i_{2}p_{1}}\sum_{k=1}^{p_{1}-i_{1}-1}%
x^{k}\ \ \ -x^{\left(  i_{1}+1\right)  p_{2}}\ \ \ -x^{\left(  i_{1}+1\right)
p_{2}}x^{i_{2}p_{1}}\sum_{k=1}^{p_{1}-i_{1}-1}x^{k} &  &  \text{by
rearranging}\\
&  =1\ \ \ +x^{i_{1}p_{2}+i_{2}p_{1}+1}\sum_{k=0}^{p_{1}-2-i_{1}}%
x^{k}\ \ \ -x^{\left(  i_{1}+1\right)  p_{2}}\ \ \ -x^{p_{2}}x^{i_{1}%
p_{2}+i_{2}p_{1}+1}\sum_{k=0}^{p_{1}-2-i_{1}}x^{k} &  &  \text{by reindexing
}k\ \text{with }k+1\\
&  =1\ \ \ +g_{i}\ \ \ -x^{\left(  i_{1}+1\right)  p_{2}}\ \ \ -x^{p_{2}}g_{i}
&  &  \text{{}}%
\end{align*}
where%
\[
g_{i}=x^{i_{1}p_{2}+i_{2}p_{1}+1}\sum_{k=0}^{p_{1}-2-i_{1}}x^{k}%
\]

\item Case: $i_{3}>i_{1}.$ Note that $\mathcal{T}_{i_{3}-i_{1}+1}\left(
x\right)  =x.$ Hence%
\begin{align*}
f_{i} &  =1\ \ \ -x^{\left(  i_{1}+1\right)  p_{2}}\ \ \ -x^{i_{1}p_{2}%
}\left(  x^{p_{2}}-1\right)  x^{i_{2}p_{1}}\sum_{k=1}^{p_{1}-i_{1}-1}%
x^{k}\ \ +x^{i_{1}p_{2}}\left(  x^{p_{2}}-1\right)  x^{i_{2}p_{1}}\sum
_{k=0}^{p_{1}-1}x^{k+1} &  &  \text{{}}\\
&  =1\ \ \ -x^{\left(  i_{1}+1\right)  p_{2}}\ \ \ -x^{i_{1}p_{2}}\left(
x^{p_{2}}-1\right)  x^{i_{2}p_{1}}\sum_{k=1}^{p_{1}-i_{1}-1}x^{k}%
\ \ +x^{i_{1}p_{2}}\left(  x^{p_{2}}-1\right)  x^{i_{2}p_{1}}\sum_{k=1}%
^{p_{1}}x^{k} &  &  \text{by reindexing }k+1\\
&&& \text{with }k\\
&  =1\ \ \ -x^{\left(  i_{1}+1\right)  p_{2}}\ \ +x^{i_{1}p_{2}}\left(
x^{p_{2}}-1\right)  x^{i_{2}p_{1}}\sum_{k=p_{1}-i_{1}}^{p_{1}}x^{k} &  &
\text{by telescoping sum}\\
&  =1\ \ \ -x^{i_{1}p_{2}}x^{i_{2}p_{1}}\sum_{k=p_{1}-i_{1}}^{p_{1}}%
x^{k}\ \ -x^{\left(  i_{1}+1\right)  p_{2}}+x^{\left(  i_{1}+1\right)  p_{2}%
}x^{i_{2}p_{1}}\sum_{k=p_{1}-i_{1}}^{p_{1}}x^{k}\ \  &  &  \text{by
distributing}\\
&  =1\ \ \ -x^{i_{1}\left(  p_{2}-1\right)  +\left(  i_{2}+1\right)  p_{1}%
}\sum_{k=0}^{i_{1}}x^{k}\ \ -\ \ x^{\left(  i_{1}+1\right)  p_{2}%
}\ \ +x^{p_{2}}x^{i_{1}\left(  p_{2}-1\right)  +\left(  i_{2}+1\right)  p_{1}%
}\sum_{k=0}^{i_{1}}x^{k} &  &  \text{by reindexing }k\ \text{with } \\%
&&&k+p_{1}-i_{1}\\
&  =1\ \ \ +g_{i}\ \ \ -x^{\left(  i_{1}+1\right)  p_{2}}\ \ \ -x^{p_{2}}g_{i}
&  &  \text{{}}%
\end{align*}
where%
\[
g_{i}=-x^{i_{1}\left(  p_{2}-1\right)  +\left(  i_{2}+1\right)  p_{1}}%
\sum_{k=0}^{i_{1}}x^{k}%
\]

\end{enumerate}

\noindent Put together we finally have%
\[
f_{i}=1+g_{i}-x^{\left(  i_{1}+1\right)  p_{2}}-x^{p_{2}}\ g_{i}%
\]%
\[
g_{i}=\left\{
\begin{array}
[c]{llll}%
\displaystyle+x^{i_{1}p_{2}+i_{2}p_{1}+1}\;\;\;\;\sum_{k=0}^{p_{1}-2-i_{1}} &
x^{k} & \text{if } & i_{3}\leq i_{1}\\
\displaystyle-x^{i_{1}\left(  p_{2}-1\right)  +\left(  i_{2}+1\right)  p_{1}%
}\sum_{k=0}^{i_{1}} & x^{k} & \text{if} & i_{3}>i_{1}%
\end{array}
\right.
\]
\noindent
\end{proof}

\subsection{Non-overlapping and Ordered}

In the previous subsection, we found an explicit expression for $f_{i}$. In
this subsection, we show that the explicit expression is non-overlapping and ordered.

\begin{lemma}
[Non-overlapping and ordered]\label{lem:no-ordered}The explicit expression%
\[
f_{i}=1+g_{i}-x^{\left(  i_{1}+1\right)  p_{2}}-x^{p_{2}}\ g_{i}%
\]
does not have overlapping of terms and that their exponents are ordered in the
ascending order.
\end{lemma}

\begin{proof}
For this, it is convenient to name the sub-polynomials in $f_{i}$ as follows:
\[
f_{i}~~=~~~\underset{S_{1}}{\underbrace{1}}~~~\underset{S_{2}%
}{\underbrace{+g_{i}}}~~~\underset{S_{3}}{\underbrace{-x^{\left(
i_{1}+1\right)  p_{2}}}}~~~\underset{S_{4}}{\underbrace{-x^{p_{2}}\ g_{i}}}%
\]
It suffices to show that $\mathrm{tdeg}(S_{j+1})-\deg(S_{j})>0\ $for
$j=1,2,3$, where $\mathrm{tdeg}$ denotes the tail (lowest) degree.\ We will
show it for the two cases: $i_{3}\leq i_{1}$ and $i_{3}>i_{1}$.

\begin{enumerate}
\item Case $i_{3}\leq i_{1}$.

Recall
\[
g_{i}=x^{i_{1}p_{2}+i_{2}p_{1}+1}\sum_{k=0}^{p_{1}-2-i_{1}}x^{k}%
\]

Note%
\begin{align*}
\mathrm{tdeg}(S_{2})-\deg(S_{1}) &  =\left(  i_{1}p_{2}+i_{2}p_{1}+1\right)
-\left(  0\right)   &  & \\
&  =i_{1}p_{2}+i_{2}p_{1}+1 &  & \\
&  >0 &  & \\
\mathrm{tdeg}(S_{3})-\deg(S_{2}) &  =\left(  \left(  i_{1}+1\right)
p_{2}\right)  -\left(  i_{1}p_{2}+i_{2}p_{1}+1+p_{1}-2-i_{1}\right)   &  & \\
&  =1+i_{1}+p_{2}-\left(  i_{2}+1\right)  p_{1} &  & \\
&  \geq1+i_{1}+p_{2}-q_{2}p_{1} &  &  \text{since }i_{2}\leq q_{2}-1\\
&  =1+i_{1}+1 &  &  \text{since }p_{2}=p_{1}q_{2}+1\\
&  >0 &  & \\
\mathrm{tdeg}(S_{4})-\deg(S_{3}) &  =\left(  p_{2}+i_{1}p_{2}+i_{2}%
p_{1}+1\right)  -\left(  \left(  i_{1}+1\right)  p_{2}\right)   &  & \\
&  =i_{2}p_{1}+1 &  & \\
&  >0 &  &
\end{align*}

\item Case $i_{3}>i_{1}$.

Recall
\[
g_{i}=-x^{i_{1}\left(  p_{2}-1\right)  +\left(  i_{2}+1\right)  p_{1}}%
\sum_{k=0}^{i_{1}}x^{k}%
\]
Note%
\begin{align*}
\mathrm{tdeg}(S_{2})-\deg(S_{1}) &  =\left(  i_{1}\left(  p_{2}-1\right)
+\left(  i_{2}+1\right)  p_{1}\right)  -\left(  0\right)   &  & \\
&  =i_{1}\left(  p_{2}-1\right)  +\left(  i_{2}+1\right)  p_{1} &  & \\
&  >0 &  & \\
\mathrm{tdeg}(S_{3})-\deg(S_{2}) &  =\left(  \left(  i_{1}+1\right)
p_{2}\right)  -\left(  i_{1}\left(  p_{2}-1\right)  +\left(  i_{2}+1\right)
p_{1}+i_{1}\right)   &  & \\
&  =p_{2}-\left(  i_{2}+1\right)  p_{1} &  & \\
&  \geq p_{2}-q_{2}p_{1} &  &  \text{since }i_{2}\leq q_{2}-1\\
&  =1 &  & \\
&  >0 &  & \\
\mathrm{tdeg}(S_{4})-\deg(S_{3}) &  =\left(  p_{2}+i_{1}\left(  p_{2}%
-1\right)  +\left(  i_{2}+1\right)  p_{1}\right)  -\left(  \left(
i_{1}+1\right)  p_{2}\right)   &  & \\
&  =\left(  i_{2}+1\right)  p_{1}-i_{1} &  & \\
&  \geq\left(  i_{2}+1\right)  p_{1}-\left(  p_{1}-2\right)   &  &
\text{since }i_{1}\leq p_{2}-2\\
&  =i_{2}p_{1}+2 &  & \\
&  >0 &  &
\end{align*}

\end{enumerate}

\noindent Hence, the explicit expression $f_{i}$ does not have overlapping of
terms and that their exponents are ordered in the ascending order.
\end{proof}

\subsection{Proof of Main result (Theorem \ref{thm:exp})}

Finally we are ready to prove the main result (Theorem \ref{thm:exp}). We will
prove it by combining several lemmas proved in the previous subsections.

\begin{proof}
[Proof of main result (Theorem\ref{thm:exp} )]Let $p_{2}\equiv
1\operatorname{mod}p_{1}$ and $p_{3}\equiv1\operatorname{mod}p_{1}p_{2}$. We
need to prove two claims $\mathsf{C1}$ and $\mathsf{C2}.$

\begin{enumerate}
\item[\textsf{C1.}] From Lemmas \ref{lem:partition} and \ref{lem:fi}, we have%
\begin{align*}
\Phi_{p_{1}p_{2}p_{3}}  &  =\sum_{i\in I}f_{i}\ x^{i\cdot\rho}%
\ \ +\ \ x^{\varphi\left(  p_{1}p_{2}p_{3}\right)  }\\
f_{i}  &  =1+g_{i}-x^{\left(  i_{1}+1\right)  p_{2}}-x^{p_{2}}\ g_{i}\\
g_{i}  &  =\left\{
\begin{array}
[c]{llll}%
\displaystyle+x^{i_{1}p_{2}+i_{2}p_{1}+1}\;\;\;\;\sum_{k=0}^{p_{1}-2-i_{1}} &
x^{k} & \text{if } & i_{3}\leq i_{1}\\
\displaystyle-x^{i_{1}\left(  p_{2}-1\right)  +\left(  i_{2}+1\right)  p_{1}%
}\sum_{k=0}^{i_{1}} & x^{k} & \text{if} & i_{3}>i_{1}%
\end{array}
\right.
\end{align*}
where%
\begin{align*}
I  &  =\left\{  0,\ldots,p_{1}-2\right\}  \times\left\{  0,\ldots
,q_{2}-1\right\}  \times\left\{  0,\ldots,p_{1}-1\right\}  \times\left\{
0,\ldots,q_{3}-1\right\} \\
\rho &  =\left[  \left(  p_{2}-1\right)  \left(  p_{3}-1\right)
,\ p_{1}\left(  p_{3}-1\right)  ,\ p_{3}-1,\ p_{1}p_{2}\right]
\end{align*}

\item[\textsf{C2.}] From Lemma \ref{lem:no-ordered}, the explicit expression
for $f_{i}$ does not have overlapping and their exponents are ordered.
Examining Lemma \ref{lem:partition}, one immediately sees that the partition
does not introduce overlapping or reordering of terms. Thus we conclude that
the above expression in $\mathsf{C1}$ does not have any overlapping of terms
and their exponents are ordered in the ascending order when $\sum_{i\in I}$ is
carried out as $\sum_{i_{1}}\sum_{i_{2}}\sum_{i_{3}}\sum_{i_{4}}$.
\end{enumerate}
\end{proof}

\bibliographystyle{unsrt}
\bibliography{../../References/all}

\begin{thebibliography}{10}

\bibitem{LE2}
A.~Lenstra.
\newblock Using cyclotomic polynomials to construct efficient discrete
  logarithm cryptosystems over finite fields.
\newblock In {\em ACISP '97 Proceedings of the Second Australasian Conference
  on Information Security and Privacy}, pages 127--138, 1997.

\bibitem{BW2005}
F.~Brezing and A.~Weng.
\newblock Elliptic curves suitable for pairing based cryptography.
\newblock {\em Designs, Codes and Cryptography}, 37:133–--141, 2005.

\bibitem{SK}
S.~Tanaka and K.~Nakamula.
\newblock {\em Pairing-Friendly Elliptic Curves Using Factorization of
  Cyclotomic Polynomials}, pages 136--145.
\newblock Pairing-Based Cryptography. Springer Berlin Heidelberg, 2008.

\bibitem{HLL}
H.~Hong, E.~Lee, and H-S. Lee.
\newblock Explicit formula for optimal ate pairing over cyclotomic family of
  elliptic curves.
\newblock {\em Finite Fields Appl}, 34:45--74, 2015.

\bibitem{BL}
D.~Bloom.
\newblock On the coefficients of the cyclotomic polynomials.
\newblock {\em Amer. Math. Monthly}, 75:372--377, 1968.

\bibitem{BE3}
M.~Beiter.
\newblock Magnitude of the coefficients of the cyclotomic polynomial $f_
  {pqr}$.
\newblock {\em The American Mathematical Monthly}, 75(4):370--372, 1968.

\bibitem{BG1}
G.~Bachman.
\newblock On the coefficients of ternary cyclotomic polynomials.
\newblock {\em J. Number Theory}, 100:104--116, 2003.

\bibitem{BZ1}
B.~Bzdega.
\newblock Bounds on ternary cyclotomic coefficients.
\newblock {\em Acta Arithmética 144(1), 5-16}, 2010.

\bibitem{GA-MO}
Y.~Gallot, P.~Moree, and R.~Wilms.
\newblock The family of ternary cyclotomic polynomials with one free prime.
\newblock {\em Involve}, 4(4), 2011.

\bibitem{VA}
R.C. Vaughan.
\newblock Bounds for the coefficients of cyclotomic polynomials.
\newblock {\em The Michigan Mathematical Journal}, 21(4):289--295, 1975.

\bibitem{VA2}
H.~L. Montgomery and R.~Vaghan.
\newblock The order of magnitude of the m-th coefficients of cyclotomic
  polynomials.
\newblock {\em Glasgow Math}, 27:143--159, 1985.

\bibitem{VA3}
P.~Erdos and R.~C. Vaughan.
\newblock Bounds for the r -th coefficients of cyclotomic polynomials.
\newblock {\em J. London Math. Soc}, 2:393–400, 1974.

\bibitem{CA1}
L.~Carlitz.
\newblock The number of terms in the cyclotomic polynomial $f_{ pq}(x)$.
\newblock {\em The American Mathematical Monthly}, 73(9):979--981, 1966.

\bibitem{BZ4}
B.~Bzdega.
\newblock Jumps of ternary cyclotomic polynomials.
\newblock {\em Acta Arithmética 163(3), 203-213}, 2014.

\bibitem{HH}
H.~Hong, E.~Lee, H-S. Lee, and C-N. Park.
\newblock Maximum gap in (inverse) cyclotomic polynomial.
\newblock {\em Journal of Number Theory}, 132:2297--2317, 2012.

\bibitem{Moree2014}
P.~Moree.
\newblock Numerical semigroups, cyclotomic polynomials, and bernoulli numbers.
\newblock {\em The American Mathematical Monthly}, 121(10):890--902, 2014.

\bibitem{Zhang2016}
B.~Zhang.
\newblock Remarks on the maximum gap in binary cyclotomic polynomials.
\newblock {\em Bull. Math. Soc. Sci. Math. Roumanie Tome}, 59(107)(1):109--115,
  2016.

\bibitem{Camburu2016}
O.-M. Camburua, E.-A. Ciolanb, F.~Lucac, P.~Moree, and I.~E. Shparlinski.
\newblock Cyclotomic coefficients: gaps and jumps.
\newblock {\em Journal of Number Theory}, 163:211--237, 2016.

\bibitem{AH2017}
H.~Hong M.~Ambrosino and E.~Lee.
\newblock Maximum gap of a certain family of ternary cyclotomic polynomials.
\newblock Technical report, arXiv 1702.07650, 2017.

\bibitem{AM2}
A.~Arnold and M.~Monagan.
\newblock A high-performance algorithm for calculating cyclotomic polynomials.
\newblock {\em Proceedings of PASCO, ACM Press}, pages 112--120, 2010.

\bibitem{AM1}
A.~Arnold and M.~Monagan.
\newblock Calculating cyclotomic polynomials of very large height.
\newblock {\em Math. Comp.}, 80:2359--2379, 2011.

\bibitem{BE1}
M.~Beiter.
\newblock The midterm coefficient of the cyclotomic polynomial $ f_{pq} (x)$.
\newblock {\em American mathematical monthly}, 71:769--770, 1964.

\bibitem{LE1}
H.~Lenstra.
\newblock Vanishing sums of roots of unity.
\newblock In {\em Proceedings (Part II) of Bicentennial Congress Wiskundig
  Genootschap (Vrije Univ. Amsterdam),}, pages 249--268, 1978.

\bibitem{LL}
T.~Y. Lam and K.~H. Leung.
\newblock On the cyclotomic polynomial $\phi_{pq} (x)$.
\newblock {\em Ame. Math. Monthly}, 103(7):562--564, 1996.

\bibitem{TH2000}
R.~Thangadurai.
\newblock On the coefficients of cyclotomic polynomials.
\newblock {\em Cyclotomic Fields and Related Topics (Pune, 1999)}, pages
  311--322, 2000.

\bibitem{Mor09}
P.~Moree.
\newblock Inverse cyclotomic polynomials.
\newblock {\em Journal of Number Theory}, 129(3):667--680, 2009.

\bibitem{BAC}
G.~Bachman.
\newblock Flat cyclotomic polynomials of order three.
\newblock {\em Bull.London Math. Soc.}, 38:53--60, 2006.

\bibitem{KA1}
N.~Kaplan.
\newblock Flat cyclotomic polynomials of order three.
\newblock {\em Journal of Number Theory}, 127:118--126, 2007.

\bibitem{KA2}
N.~Kaplan.
\newblock Flat cyclotomic polynomials of order four and higher.
\newblock {\em Integers}, 10:357--363, 2010.

\bibitem{AK}
Ala'a Al-Kateeb.
\newblock {\em Structures and Properties of Cyclotomic polynomials}.
\newblock PhD thesis, North carolina state university, 2016.

\end{thebibliography}
{}

\end{document}